\documentclass[11pt]{article}
\usepackage{latexsym}
\usepackage{amsmath}
\usepackage{amsfonts}
\usepackage{amssymb}
\usepackage{graphicx}
\usepackage{appendix}
\usepackage{enumerate}
\usepackage{enumitem}
\usepackage{bbm}
\usepackage{bm}
\newcommand{\bnu}{\boldsymbol{\nu}}

\usepackage[colorlinks,linkcolor=blue,anchorcolor=blue,citecolor=blue,backref=page, pdfencoding=auto]{hyperref}

\renewcommand*{\backref}[1]{}
\renewcommand*{\backrefalt}[4]{%
\ifcase #1 (Not cited.)%
\or        (p.\,#2)%
\else      (pp.\,#2)%
\fi}

\newenvironment{proof}{{\bf Proof:  }}{\hfill\rule{2mm}{2mm}}

\newtheorem{thm}{Theorem}[section]
\newtheorem{cor}[thm]{Corollary}
\newtheorem{defn}[thm]{Definition}
\newtheorem{exmpl}[thm]{Example}
\newtheorem{lem}[thm]{Lemma}
\newtheorem{prop}[thm]{Proposition}
\newtheorem{rem}[thm]{Remark}

\newcommand{\ds} {\displaystyle}

\newcommand{\equaldef}[0]{{\buildrel def \over =}}
\newcommand{\R}[0]{\mathbb{R}}
\newcommand{\T}[0]{\mathbb{T}}
\newcommand{\Cc}[0]{\mathbb{C}}
\newcommand{\N}[0]{\mathbb{N}}
\newcommand{\Q}[0]{\mathbb{Q}}
\newcommand{\Z}[0]{\mathbb{Z}}

\newcommand{\1}[0]{\mathbb{I}}
\newcommand{\K}[0]{\mathbb{K}}
\newcommand{\B}[0]{\mathcal{B}}
\newcommand{\E}[0]{\mathcal{E}}
\newcommand{\F}[0]{\mathcal{F}}
\newcommand{\bE}[0]{\textbf{E}}
\newcommand{\bmu}[0]{\bm \mu}

\newcommand{\sgn}[0]{\textrm{Sg}}
\newcommand{\bml}[0]{\bm \lambda}
\newcommand{\eps}[0]{\varepsilon}

\newcommand{\qid}[0]{\enspace\lower6pt\vbox{\hrule height2.6pt depth-
2pt\hbox{\vrule
width0.6pt\kern1pt\vbox{\kern1pt\phantom{j}\kern1pt}\kern1.5pt
\vrule width0.6pt}\hrule height2.6pt depth-2pt}}

\def\build#1_#2^#3{\mathrel{\mathop{\kern 0pt#1}\limits_{#2}^{#3}}}
\def\tend#1#2{\build\hbox to 12mm{\rightarrowfill}_{#1\rightarrow #2}^{ }}

\def\endproof{\hfill{\vrule height4pt width6pt depth2pt}}
\newcommand{\setdef}{\stackrel {\rm {def}}{=}}

\title{Weakly tame systems, their characterizations and application}

\author{{\small el Houcein el Abdalaoui} \\
{\small Department of Mathematics} \\         
{\small Universite of Rouen Normandy, France} \\              
{\medskip} \\
{and} \\
{\medskip} \\
{\small Mahesh Nerurkar} \\
{\small Department of Mathematics} \\
{\small Rutgers University, Camden NJ 08102 USA}}

\date{}

\begin{document}
\maketitle
\renewcommand{\theequation}{\thesection.\arabic{equation}}
\numberwithin{equation}{section}

\medskip

\begin{abstract}

We explore the notion of discrete spectrum and its various characterizations for ergodic measure preserving
actions of an amenable group on a compact metric space. We introduce a notion of ‘weak-tameness’, which
is a measure theoretic version of a notion of ‘tameness’ introduced by E. Glasner, based on the work of 
A. K\"{o}hler \footnote{A. K\"{o}hler  introduced this notion and call such systems ``regular".}, and characterize such topological dynamical systems as systems for which every invariant measure
has a discrete spectrum. Using the work of M. Talagrand, we also characterize weakly tame as well as tame systems
in terms of the notion of ‘witness of irregularity’ which is based on ‘up-crossings’. Then we establish that ‘strong
Veech systems’ are tame. In particular, for any countable amenable group $T$, the flow on the orbit closure of the
translates of a ‘Veech function’ $f \in \K(T)$ is tame. Thus Sarnak’s Möbius orthogonality conjecture holds for this
flow and as a consequence, we obtain an improvement of Motohashi-Ramachandra 1976’s theorem on the Mertens function in  short interval. We further improve Motohashi-Ramachandra’s bound to $1/2$ under Chowla conjecture.

\smallskip


\smallskip

\end{abstract}

\pagenumbering{arabic}

\section{Preliminaries and notation}

\medskip

\medskip

This note results from trying to understand whether the notion of `discrete spectrum' of a compact,
metric, ergodic dynamical system $(X,T,\mu )$ can be captured in terms of the regularity properties of
the elements its enveloping semigroup. It turns out that even though in general this type of
characterization of systems with discrete spectrum is not possible, our study allows us to obtain
other characterizations for more general acting groups $T$. In the second section we shall introduce
the notion of `$\mu$-tameness', which is a weakening of the notion of `tameness' introduced by
E. Glasner \cite{G} based on the work of A. K\"{o}hler \cite{Kh} (such systems in \cite{Kh} are called ``regular".). This notion also exists in a
dormant form in the work of Bourgain \cite{B76}. The third section is devoted to the study of
$\mu$-mean equicontinuity. In the fourth section we shall also characterize `weakly-tame' and
tame systems using the work of M. Talagrand on Glivenko-Cantelli families and the notion of
`witness of irregularity'. This section can be viewed as our efforts to give a simple and aseptic exposition,
in the language of dynamical systems, to the results generated by the 1976 paper of J. Bourgain and the 1987 paper
of M. Talagrand. These two papers are `difficult to digest' and have generated a huge amount of
literature.

In the fifth section we study -- what we shall call `Veech systems'. Professor W. Veech introduced an interesting class $\K(\Z)$ of functions on integers which properly contains the class of all weakly
almost periodic functions. Translation flow on the orbit closure of such a function is an example of
a Veech system. The final section is where we apply our results to translations flow of a special Veech
function will allow us to improve `Motohashi-Ramachandra estimates' of the Mertens function in short
interval. Thus, our method yields a `dynamical approach' to these type of number theoretic
estimates. As another application, we also establish, by a simple observation of Rauzy \cite{Rauzy},
Besicovich almost periodicity of a number theoretic function arsing from the $B$-free integers integers.

\smallskip

We begin by introducing the notation and basic definitions. By a topological dynamical system $(X,T)$
we mean a compact, Hausdorff space $X$ on which a topological group $T$ acts (on the right), with a
jointly continuous action $(x,t)\to \pi (x,t) \equiv \pi_t(x) \equiv xt$, $x\in X$ and $t\in T$. In what
follows topology of $T$ will not play any part and so one may as well assume $T$ to be discrete. The set 
$\mathcal{O}(x) = \{xt\ |\ t\in T\}$ is the orbit of $x\in X$. A subset $M\subset X$ is invariant
if $\mathcal{O}(x)\subset M$ for all $x\in M$. System $(X,T)$ is point transitive if it has a dense
orbit and is minimal if all orbits are dense, (equivalently there are no proper closed invariant sets).

Following Prof. R. Ellis's algebraic approach to dynamics, we try to capture the asymptotic properties
of the system in terms of the algebraic properties of a suitable compactification of the acting group $T$.
We begin by introducing three important compactifications we need, (1) the Stone-\u{C}ech compactification
$(\beta T,T)$, (2) the enveloping or Ellis semigroup $E(X,T)$ and (3) the `ergodic analog' of $E(X,T)$-namely
$(\Omega_\mu ,T)$. As the notation indicates, all of these compactifications are themselves going to be
topological dynamical systems where the underlying compactification will be a compact Hausdorff space with
a semigroup structure which has the common additional property of being an $\mathcal{E}$-semigroup.
Before describing the compactifications,  we recall the definition of an $\mathcal{E}$-semigroup.

\smallskip

\begin{defn} A set $E$ is an  $\mathcal{E}$ semigroup if (i) it is a semigroup, (ii) it has a compact,
Hausdorff topology and (iii) in this topology the left multiplication map $L_p:E\to E$, $L_p(q) = pq$,
$p,q\in E$ is continuous.
\end{defn}

\smallskip

\noindent {\bf (1) Stone-\u{C}ech Compactification $\beta T$} : Recall that element of $\beta T$ are
ultrafilters on $T$. In fact $\beta T$ is an $\mathcal{E}$-semigroup with multiplication of ultrafilters
$p,q\in \beta T$ given by
$$
A\in pq\ \text {if and only if}\quad A*p\in q\,, \quad \text {where}\ A*p = \{t\in T\ |\ At^{-1}\in q\}\,.
$$
This compactification also has the universality property that any continuous map from $T$ to any
compact, Hausdorff space has a unique continuous extension to $\beta T$, (here $T$ has the discrete
topology). This universal property allows one to extend the $T$ action on a compact, Hausdorff space $X$
to an action of the semigroup $\beta T$. In particular, if a net $\{t_\alpha\}$ in $T$ converges to
$p\in \beta T$, then $xp = \lim\limits_{\alpha}xt_\alpha$, for $x\in X$. Furthermore, this also implies
that the dynamical system $(\beta T ,T)$ is a universal point-transitive system..

\smallskip

\noindent {\bf (2) Enveloping Semigroup $E(X,T)$}: Let $E(X,T) = \overline {\{\pi_t\ |\ t\in T\}}$, where the
closure is in the topology of pointwise convergence on all maps from $X$ to $X$. Then $E(X,T)$ itself
is an $\mathcal{E}$-semigroup and $(E(X),T)$ is a point transitive dynamical system.

\smallskip

\noindent {\bf (3) Measure theoretic enveloping semigroup $(\Omega_\mu,T)$}: Let $(X,T)$ be a compact,
metric dynamical system with a $T$ invariant Borel probability measure $\mu$ on $X$. Let $H = L^2(X,\mu)$ and
let $U_t[f] = [f_t]$ where for a measurable function $f:X \to \Cc$, $[f]$ denotes its equivalence class (mod $\mu$) and
$f_t(x) = f(xt)$. Then $t\to U_t$ unitary representation of $T$ on $H$. It is important that we distinguish 
between a measurable function $f$ and its equivalence class $[f]$. Let $\Omega_\mu = \overline {\{U_t\ |\ t\in T\}}$,
where the closure is in the weak operator topology. Then $\Omega_\mu$ is a $\mathcal{E}$-semigroup and
$(\Omega_\mu ,T)$ is itself a point transitive dynamical system. The dynamical system $(\Omega_\mu ,T)$ is weakly
almost periodic (see \cite{EN1}, \cite {EN2}). We shall list its special properties shortly.

\smallskip

Next we recall a few general facts about $\mathcal{E}$-semigroups, (see also \cite{A}, \cite{E}). Let $E$ be an
$\mathcal{E}$-semigroup. A subset $M\subset E$ is a right ideal if it is closed and $m\in M, e\in E$ implies $me\in M$.
The following lemma summarizes the structure of minimal right ideals.

\smallskip

\begin{prop} Let $E$ be a $\mathcal{E}$ semigroup and $M\subset E$ be a minimal right ideal of $E$. Then
\begin{enumerate}[label=(\arabic*)]
\item The set $J_M = \{v\in M\ |\ v^2 = v\}$ of idempotents in $M$ is non-empty.
\item  For each $v\in J_M$, the set $Mv$ is a subgroup of $M$ with identity $v$.
\item  $vp = p$, for each $v\in J_M$ and $p\in M$, i.e. each $v\in J_M$ is a left identity in $M$.
\item  Any two minimal right ideals of $E = \beta T$ are isomorphic.
\end{enumerate}
\end{prop}

\smallskip

As mentioned above, all of the previous three examples are $\mathcal{E}$-semigroups but $\Omega_\mu$ has many
additional features which we list now, (see (\cite{EN1}, \cite{EN2} for proofs).

\smallskip

\begin{prop}\label{wap} (1) The flow $(\Omega_\mu ,T)$ is weakly almost periodic, in particular,
\begin{enumerate}[label=(\arabic*)]
\addtocounter{enumi}{1}
\item The right multiplication $R_p(q) = qp$, $p,q\in \Omega_\mu$ is also continuous.
\item There is only one minimal right ideal in $\Omega_\mu$, which we denote by $I_\mu$.
\item The ideal $I_\mu$ has a unique idempotent, which we denote by $P_\mu$ and 
which commutes with all elements of $I_\mu$.
\item The ideal $I_\mu$ is closed under $*$-the operator adjoint,
\item In fact $I_\mu$ is a compact topological group of operators and the weak and
strong operator topologies on $I_\mu$ coincide.
\end{enumerate}

\end{prop}

\smallskip

\begin{rem} Let $\nu$ be the normalized Haar measure on the compact topological group $I_\mu$
and let $ C_\mu =\displaystyle \int_{I_\mu}gd\nu$. Then the operator $C_\mu$ is the projection on $T$ invariant functions. Thus
\begin{enumerate}[label=(\roman*)]
\item $(X,T,\mu)$ is ergodic iff $C_\mu =C$-the the projection on constants,
\item $(X,T,\mu)$ has discrete spectrum if and only if $P_\mu = I$-the identity operator
and in this case $\Omega_\mu = I_\mu$.
\item  $(X,T,\mu)$ is weakly mixing if and only if $P_{\mu} = C$, (see \cite{EN2} for details).
\end{enumerate}

\end{rem}

\smallskip

\paragraph{\bf The projection maps $p\to \rho_p : \beta T \to E(X,T)$ and $p\to U_p :\beta T\to \Omega_\mu$.}

\smallskip
\begin{enumerate}[label=(\alph*)]
\item Since $(\beta T ,T)$ is a universal point transitive flow and $E(X,T)$ is point transitive, there
is a canonical factor map $p\to \rho_p : \beta T \to E(X,T)$ such that $\rho_e = i_X$, i.e. this maps the identity
$e$ of $T$ to the identity map $i_X$ on $X$. Equivalently, given a point transitive flow $(X,T,x_0)$, there is
a unique continuous extension to $\beta T$ of the map $t\to x_0t:T\to X$. This defines a $\beta T$ action on $X$ given by
$$
x\cdot p = \rho_p(x)\,, \quad (x\in X\,,\ p\in \beta T)\,.
$$
\item  Again, since $(\beta T,T)$ is a universal point transitive flow, the map
$$
p\to U_p :\beta T\to \Omega_{\mu}\,,
$$
is the unique continuous extension of the map $t\to U_t: T\to \Omega_{\mu}$ that takes the identity of $T$ to $I$- the 
identity operator. It is also a semigroup homomorphism. Thus, if $\{t_\alpha\}$ is a net in $\beta T$ such that
$t_\alpha \to p$ in $\beta T$, then
$$
U_p = \lim\limits_{t_\alpha \to p}U_{t_\alpha}\,.
$$
Now fix a minimal (right) ideal $M\subset \beta T$, (which ideal hardly matters because they are all
isomorphic). Then $M$ is a closed, $T$ invariant set of $(\beta T,T)$. Hence it is also a minimal set of
the dynamical system $(\beta T,T)$. Since $(\beta T,T)$ is a universal point transitive flow, it follows that $(M,T)$
is a universal minimal flow. Thus, the restriction of the above map gives a canonical projection
$$
p\to U_p : M\to I_{\mu}\subset \Omega_{\mu}\,.
$$
Note that since $\Omega_\mu$ has a unique minimal set $I_{\mu}$, all minimal ideals will project onto $I_{\mu}$ and
since $I_{\mu}$ is a group, all idempotents in any minimal ideal will be mapped into the projection operator $P_\mu$
which
is the identity of $I_\mu$. Thus, given a net $t_\alpha \to p$ in $M$,
$$
U_p = \lim\limits_{t_\alpha \to p}P_{\mu}U_{t_\alpha} = \lim\limits_{t_\alpha \to p}U_{t_\alpha}P_{\mu}\,.
$$
Note that $U_v = P_\mu$ for all idempotents $v\in M$.

\smallskip

Now, given a measurable map $f:X \mapsto \Cc$, let $[f]$ denote its equivalence class determined by the relation
defined by equality modulo a set of $\mu$ measure zero. Then for $[f]\in L^2(X,\mu )$ and $p\in \beta T$,  we set
$$
[f]_p = U_p[f]\,.
$$

\end{enumerate}

\smallskip

\begin{rem}{$\quad$}
	
\begin{enumerate}[label=(\arabic*)]
\item Even though $[f]_t = [f_t]$ for $t\in T$, we cannot replace $t\in T\subset \beta T$ by a general $p\in \beta T$
in this equation. To begin with, in general $f_p$ may not be even measurable, so $[f_p]$ makes no sense. Even in the
special case when $\rho_p = i_X$, obviously $f(xp) = f(x)$ but even in this case we cannot say $U_p[f] = [f]$, as the
following example will show.
\item Note that for the transformation $T(x,y) = (x+\alpha ,x+y)$ on the $2$-torus ${\mathbb{T}}^2$, (where
$\alpha \notin \Q$), $\rho_v = i_X$ for all idempotents $v\in \beta T$, (since $T$ is distal), and $U_v = P_{\mu} \ne I$,
for all minimal idempotents $v\in M$, (where $\mu$ is the usual Lebesgue measure on $\mathbb{T}^2$).
\item For systems with discrete spectrum, if $v = v^2\in M$, then $U_v = I$. However $U_v = I$ may not imply
$\rho_v = i_X$ as the following simple example shows. Let $(X,T)$ be a minimal Sturmian shift which is always uniquely
ergodic and has discrete spectrum with respect to the unique invariant measure $\mu$. Then $U_u = I$, for any minimal
idempotent in $M$ but $\rho_u \ne i_X$ for some $u = u^2\in M$, as there are non-trivial proximal pairs in the system.
\end{enumerate}	
\end{rem}

\smallskip

\section{\texorpdfstring{$\mu$-compact vectors and $\mu$-tame vectors.}{}}

\smallskip

\begin{defn} Let $(X,T,\mu)$ be a compact metric, ergodic dynamical system. A function $f\in L^2(X,\mu )$
is a compact vector if the orbit $\{U_tf\ |\ t\in T\}$ of $f$ has compact closure in the norm topology on
$L^2(X,\mu )$.
\end{defn}

\smallskip

With this definition, the following is a corollary to Proposition \ref{wap}.

\smallskip

\begin{prop} Let $(X,T,\mu )$ be a compact metric system and let $f\in L^2(X,\mu )$.
Then the following statements are equivalent,
\begin{enumerate}[label=(\arabic*)]
\item $f$ is a compact vector,
\item  $P_\mu (f) = f$,
\item  the weak and the strong topologies on the set $\overline {\mathcal{O}(f)}$ of orbit closure of $f$ coincide,
\item  the system $(\overline {\mathcal{O}(f)},T)$ is minimal,
\item  for some $m\in M$, (where $M$ is some minimal right ideal of $\beta T$), there is a sequence $\{t_n\}$ in $T$
such that $U_{t_n}[f] \equiv  [f]_{t_n}\to U_m[f]$ in $L^2(X,\mu)$. 
\end{enumerate}
\end{prop}

\smallskip

Next, we introduce the notion of a $\mu$-tame function.

\smallskip

\begin{defn} 
\begin{enumerate}[label=(\alph*)]
\item Let $f\in L^2(X,\mu )$. Then $f$ is said to be $\mu$-tame if there exists a $q\in \beta T$,
a Borel set $N\subset X$ with $\mu (N) = 0$ and a sequence $\{t_n\}$ in $T$ such that
\begin{enumerate}[label=(\arabic*)]
\item  $f_{t_n} \to f_q$ pointwise on $X\backslash N$. Thus, in particular the map
$f_q \equiv f\circ \rho_q :X\backslash N :\to \Cc$ is a Borel map and
\item  $U_q\in I_\mu$. This will imply that there exists some $m\in M$ such that
$$
U_m[f] = [1_{X\backslash N}f_q] = U_q[f]\,,
$$
where $1_{X\backslash N}$ is the indicator function of the set $X\backslash N$.
\end{enumerate}
\item System $(X,T,\mu)$ will be called $\mu$-tame if each $f\in L^2(X,T)$ is a $\mu$-tame vector.
\end{enumerate}	
\end{defn}

\smallskip

\begin{rem} We recall the notion of a tame dynamical system $(X,T)$ introduced by Glasner-Kh\"{o}ler, (see \cite{G}, \cite{Kh}).

\smallskip

\begin{defn} A compact, Hausdorff topological dynamical system $(X,T)$ is tame if each element of $E(X,T)$ is a
Baire-1 class function.
\end{defn}

\smallskip

It follows that if $(X,T)$ is a tame system then $(X,T,\mu )$ is a $\mu$-tame system for any invariant
Borel probability measure $\mu$ on $X$.
\end{rem}

\smallskip

\begin{prop} Let $f\in L^2(X,\mu)$. Then $f$ is $\mu$-compact if and only if it is $\mu$-tame.
\end{prop}
\begin{proof} Suppose $f$ is $\mu$-compact. Pick any $m\in M$, where $M$ is any minimal right ideal in $\beta T$. Then
select a sequence $\{t_n\}$ in $T$ such that $U_{t_n}[f]\to U_m[f]$ in $L^2(X,\mu )$. This implies that
by passing to a subsequence, (which we again denote by $\{t_n\}$), we can assume that $f_{t_n}$ converges pointwise on a
set $X\backslash N$ for some Borel set $N$ with $\mu (N) = 0$. Now viewing the sequence $\{t_n\}$ as a net in $\beta T$,
we can find a convergent subnet  (which may not be a subsequence), converging to some $q\in \beta T$. It follows that
the pointwise limit of $f_{t_n}$ on $X\backslash N$ is a Borel function on $X\backslash N$ and equals $f_q$.
Note that $U_q[f] = U_m[f]\in I_\mu$.

Conversely, the hypothesis implies that for some sequence $\{t_n\}$ in $T$, $U_{t_n}[f]\to U_m[f]$
for some $m\in M$. Hence by (5) of Proposition (2.2), $f$ is $\mu$-compact.
\end{proof}

\smallskip

\begin{rem} It is immediate that if $f$ is tame then it is $\mu$-tame for any invariant Borel probability
measure $\mu$ on $(X,T)$. In particular, with respect to any invariant measure, a tame system $(X,T)$ has
discrete spectrum. The following example shows that system may not be tame even if all invariant ergodic measures have
discrete spectrum.
\end{rem}

\smallskip

\begin{exmpl}\label{Namioka} Consider the system on the $2$-torus $X = {\mathbb T}^2$ given by $T(x,y) = (x,x+y)$. Note that
any ergodic measure for this system is of the form $\delta_x\times \nu$, where $\nu$ is either the Lebesgue measure
$\lambda$ on the unit circle $\mathbb{T}^1$ if $x$ is irrational, or the uniform probability on the finite orbit
$y+nx\ mod\ 1$ if $x$ is rational. It follows from Namioka's work, (see (\cite {N})) that for this system
$E(X,T) = \{id \times f\ |\ f:\mathbb{T}\to \mathbb{T} \ \text {is any homomorphism}\}$. Thus, a `large number'
of elements of $E(X,T)$ are not even measurable. Thus this system is not tame. But it is easy to check that
it is $\mu$-tame for any invariant ergodic measure.

This simple example also illustrates that $\mu$-tame with respect to all invariant ergodic measures does not imply
$\mu$-tame for any invariant measure. Since $\lambda \times \lambda$ is a non-ergodic invariant measure,
with respect to which $T$ does not have discrete spectrum, this example also shows that systems can have discrete
spectrum with respect to all invariant ergodic measure but may fail to have discrete spectrum with respect to all invariant measures.
\end{exmpl}

\smallskip

\begin{rem} E. Glasner proved that if a distal system is tame, then it is equicontinuous. The above example being distal, shows
that the analogue of Glasner's result is false if `tame' is replaced by `$\mu$-tame'. However some analogue of this might be true.
For example, we would like to know if $(X,T,\mu )$ is minimal, distal and $\mu$-tame, then is it equicontinuous? and for such
non-minimal systems we ask whether the system is equicontinuous on the support of $\mu$.
\end{rem}

\smallskip

\section{\texorpdfstring{$\mu$-mean equicontinuous vectors.}{}}

\smallskip

We first recall a few necessary things about amenable groups. Let $T$ be a countable (discrete) group.

\smallskip

\begin{defn}$\quad$
\begin{enumerate}[label=(\arabic*)]	
\item Given finite sets $F\,,K\subset T$, $F$ is $(K,\epsilon )$-invariant if
$\big|KF\Delta F\big| < \epsilon \big|F\big|$.

\item A sequence $\{F_n\}$ of finite subsets of $T$ is a F\o{}lner sequence if given
any $\epsilon > 0$ and a finite set $K\subset T$, there exists a $n_0\in \N$ such that
$F_n$ is $(K,\epsilon )$-invariant for all $n>n_0$. This is equivalent to saying that
$$
\lim\limits_{n\to \infty}\frac {\big| tF_n\Delta F_n\big|}{\big|F_n\big|} = 0\,,\quad \text {for each}\ t\in T\,.
$$

\item A F\o{}lner sequence $\{F_n\}$ is tempered if there exists a constant $C>0$ such that
$$
\big|\bigcup\limits_{k\leq n}F_k^{-1}F_{n+1}\big| \leq C\big|F_{n+1}\big|\,,\quad \text {for all}\ n\in \N\,.
$$
\end{enumerate}
\end{defn}

\smallskip

\begin{rem} Every F\o{}lner sequence has a tempered subsequence, (see \cite {L}).
\end{rem}

\smallskip

\begin{defn} Fix a F\o{}lner sequence $\mathcal{F} = \{F_n\}$. Let $A\subset T$.
\begin{enumerate}[label=(\arabic*)]	
\item The (asymptotic) density $\bar {d_\mathcal{F}}$ of $A$ with
respect to $\mathcal{F}$ is given by
$$
\bar {d_\mathcal{F}} = \limsup\limits_{n\to \infty}\frac {|A\cap F_n|}{|F_n|}\,.
$$

\item For $A\subset T$ and a finite subset $F\subset T$, set
\begin{align}
D^*_F(A) & = \sup\limits_{t\in T}\frac {\big| A\cap Ft\big|}{\big |F\big|}\,, \notag \\
D^*(A) & = \inf \Big\{D^*_F(A) \ |\ F\subset T\,,\ \big|F\big|<\infty\Big\}\,. \notag 
\end{align}
Then $D^*(A)$ is called the upper Banach density of $A$.
\end{enumerate}
\end{defn}

\smallskip

\begin{lem} $\quad$
\begin{enumerate}[label=(\arabic*)]	
\item Let $\{F_n\}$ be a F\o{}lner sequence in $T$ and $A\subset T$. Then
$$
D^*(A) = \lim\limits_{n\to \infty}D^*_{F_n}(A)\,.
$$
In particular the above limit exist and is independent of the choice of F\o{}lner sequence.

\item Furthermore
$$
D^*(A) = \sup\limits_{\mathcal{F}}\big( \limsup\limits_{n\to \infty}D^*_{F_n}(A)\big)\,,
$$
where $\mathcal{F} = \{F_n\}$ varies over all F\o{}lner sequences in $T$.
\end{enumerate}

\end{lem}

\smallskip

We shall use the following pointwise ergodic theorem for $L^1$ functions, (see \cite {L}).

\smallskip

\begin{thm} Let $(X,T,\mu )$ be an ergodic, probability  preserving system with $T$ amenable.
Let $\mathcal{F} = \{F_n\}$ be a tempered F\o{}lner sequence and $f\in L^1(X,\mu )$. Then
$$
\lim\limits_{n\to \infty}\frac {1}{|F_n|}\sum\limits_{t\in F_n}f(xt) = \int_X f(x)d\mu (x)\,, \quad \text {a.e.}\ x.
$$
\end{thm}

\smallskip

\begin{defn} [Besicovitch seminorm and Besicovitch functions on $T$] Fix a tempered F\o{}lner sequence $\mathcal{F} = \{F_n\}$.
\begin{enumerate}[label=(\arabic*)]	
\item On the space of complex valued maps on $T$, define the Besicovitch seminorm $||\ ||_{B_1}$ by setting,
$$
\big\|f\big\|_{B_1} = \limsup\limits_{n\to \infty}\frac {1}{\big|F_n\big|}\sum\limits_{t\in F_n}|f|(t)\,.
$$

\item  A map $f:T\to \Cc$ is Besicovitch if $||f||_{B_1} < \infty$.
\end{enumerate}
\end{defn}

\smallskip

\begin{rem} Next, we want to define the notion of `Besicovitch almost periodic function'. When $T$ is abelian, the classical
definition says $:T\to \Cc$ is Besicovitch almost periodic if given any $\eps > 0$, there exists a trigonometric polynomial $P$ such that
$\big\|f-P\big\|_{B_1} < \eps$. By a trigonometric polynomial we mean a finite linear combination of characters of $T$. For non-abelian groups
`trigonometric polynomials' will have to be replaced by the matrix coefficient functions of finite dimensional irreducible, unitary
representations of $T$. For non-abelian $T$ these irreducible, unitary representations are not necessarily one dimensional and
hence one cannot just add the matrix coefficients functions and demand $f$ be approximated by them. A proper way to do this
is to consider the given $f$ as an element of the Hilbert space $l^2(T)$, decompose the left regular representation of $T$,
assume that it decomposes in to an orthogonal direct sum of finite dimensional irreducible unitary representations and then
demand that the projection of $f$ on each irreducible subspace be approximable by a vector valued map on $T$ with coefficients
given by the matrix coefficients of the underlying `piece of the unitary representation' from the decomposition.  

Now given a compact, metric ergodic dynamical system $(X,T,\mu )$ with amenable $T$ and $f\in L^2(X\mu )$, a.e $x\in X$ we get a
complex valued function $\psi_{x,f}(t) = f(xt)$. Since we shall be interested in maps on $T$ arising this way, we may make the above
notion precise by considering the unitary representation $t\to U_t$ on $L^2(X,\mu)$ instead of the left regular representation and
try to see when $\psi_{x,f}$ is `Besicovitch almost periodic' for almost all $x\in X$. As one can guess, this is exactly the case
when $f$ is $\mu$-compact. The next lemma puts all of this discussion on a more formal footing. 
\end{rem}

\smallskip

\begin{lem} Let $(X,T,\mu )$ be ergodic and $f\in L^2(X,\mu )$ be $\mu$-compact. Consider the closed subspace $H\equiv H_f$ of
$L^2(X,\mu )$ generated by the span of $\{U_t[f]\ |\ t\in T\}$. Then $H$ can be written as an orthogonal direct sum $H = \bigoplus V_k$,
where each $V_k$ is a finite dimensional, $U_t$ invariant subspace. The representation $U_t$ restricted to each $V_i$ is irreducible.
Let $P_if$ denote the orthogonal projection of $f$ onto $V_i$, ($i\in \N$). Then $f = \sum\limits_{i\in \N}P_if$ and each $P_if$
is of the form
$$
P_if(x) = \sum\limits_{j=1}^{d_i}\langle f^i_j(x)\,,w^i_j\rangle w^i_j\,,
$$
where $d_i = dim(V_i$, $\{w^i_j\ |\ 1\leq j\leq d_i\}$ is a fixed basis of $V_i$ and $f^i_j : X\to V_j\equiv {\Cc}^{d_j}$ are measurable maps.
\end{lem}
\begin{proof} This is just an application of the Peter-Weyl theorem to te compact topological group $I_\mu$-the unique minimal
ideal of $\Omega_\mu$. Note that, since $f$ is $\mu$-compact, $P_{\mu} f = f$ and $H$ is the closed linear space of $\{Uf\ |\ U\in I_\mu\}$.
Thus the compact topological group of unitary operators $I_\mu$ has a natural unitary representation on $H$. By Peter-Weyl theorem
$H = \bigoplus V_k$ where each $V_k$ is a finite dimensional, $U_t$ invariant subspace. The representation $U_t$ restricted to
each $V_i$ is irreducible. The rest of the lemma is a trivial consequence of linear algebra. 
\end{proof}

\smallskip

\begin{rem}$\quad$
	
\begin{enumerate}[label=(\arabic*)]		
\item The above representations of $f$ and $P_if$ are to be understood as an expressions in $L^2(X,\mu )$.

\item Note that if $f$ is $\mu$-compact and $f$ has the above representation, then
$$
U_tf (x) = \sum\limits_{i\in \N}U_tP_if =  \sum\limits_{i\in \N}\sum\limits_{j=1}^{d_i}\langle U_tf^i_j(x)\,,w^i_j\rangle w^i_j\,.
$$
Again, this representation is to be understood as an expression in $L^2(X,\mu )$. 
\end{enumerate}
\end{rem}

\smallskip

\begin{defn} $\quad$
	
\begin{enumerate}[label=(\arabic*)]	
	
\item A function which is a finite sum of functions of the form $t\to \langle U_tv,w\rangle$ will be called `generalized trigonometric
polynomials on $T$, where $t\to U_t$ is a unitary representation of $T$ on a finite dimensional vector space $V$ and $v,w\in V$.

\item A function $f\in l^2(T)$ that can be approximated in the $\big\|~\big\|_{B_1}$ norm by a generalized trigonometric polynomial will be called 
a Besicovitch almost periodic function.
\end{enumerate}
\end{defn}

\smallskip

The following theorem is a generalization to ergodic amenable group actions of a known characterization of discrete spectrum for abelian group actions.

\smallskip

\begin{thm}\label{compact-vector} Let $(X,T,\mu )$ be an ergodic system, with $T$ amenable. Then the following statements are equivalent.
\begin{enumerate}[label=(\arabic*)]	
\item  A vector $f\in L^2(X,\mu )$ is a $\mu$-compact vector.
\item  For $\mu$-almost all $x\in X$, the map $\psi_{x,f}(t) = f(xt)$ is a Besicovitch almost periodic function, in the sense that given
$\eps >0$, there exists a measurable map $P:X\to \Cc$ such that (i) for almost all $x$ the map $t\to P(xt) \equiv U_tP$ is a generalized trigonometric
polynomial and (ii) $||\psi_{x,f} - \psi_{x,P}||_{B_1} < \eps$.
\end{enumerate}
\end{thm}
\begin{proof} (1) implies (2): Since $f$ is $\mu$-compact, we have the representation
$$
\psi_{x,f} = U_tf (x) = \sum\limits_{i\in \N}U_tP_if =  \sum\limits_{i\in \N}\sum\limits_{j=1}^{d_i}\langle U_tf^i_j(x)\,,w^i_j\rangle w^i_j\,.
$$
Given $\eps >0$, select $k\in \N$ such that $\big\|f - P\big\|_2 < \eps$, (and hence $||f - P||_1 < \eps$), where
$P(x) = \sum\limits_{i=1}^kP_if(x) = \sum\limits_{i=1}^k\sum\limits_{j=1}^{d_i}\langle f^i_j(x)\,,w^i_j\rangle w^i_j$. But, by the ergodic
theorem, for almost all $x\in X$ we have $||\psi_{x,f}-\psi_{x,P}||_{B_1} = ||f - P||_1 < \eps$. This proves that for almost all $x\in X$, the map
$t\to \psi_{x,f}(t)=f(xt)$ is Besicovitch almost periodic.\\

\noindent (2) implies (1): It follows from our assumption that given $\eps > 0$, $||\psi_{x,f} - \psi_{x,P}||_{B_1} < \eps$ for almost all $x$, where $t\to U_tP$
is a generalized trigonometric polynomial. Again by the ergodic theorem $\big\|f-P \big\|_1 = \big\|\psi_{x,f}-\psi_{x,P}\big\|_{B_1} < \eps$, (for suitable $x$'s).
since $t\to U_tP$ is a generalized trigonometric polynomial, $P$-is a $\mu$-compact vector. Thus, we have shown that there is a sequence $P_n$
of $\mu$-compact vectors that converge to $f$ in the $L^1(X)$ norm. But then there is a subsequence of $\{P_n\}$ that converges pointwise almost everywhere
and hence in the $L^2(X)$ norm to $f$. Whence, $f$ is $\mu$-compact. 
\end{proof}

\smallskip

\begin{cor} \label{Besicovitch} Let $(X,T,\mu )$ be uniquely ergodic with discrete spectrum and $f\in C(X)$. Then $t\to f(xt)$ is
Besicovitch almost periodic for every $x\in X$.
\end{cor}
\begin{proof} This follows from the argument used in (1) implies (2) of the above theorem,  since each $x\in X$ is $(f,\mu )$ generic.
\end{proof}

\smallskip

Next, we generalize to amenable group actions, another characterization of $\mu$-compact vectors in terms of
$\mu$-mean equicontinuous vectors. This result is originally due to B. Scarpellini, (see \cite {S}) and more
recently to  Garc\'{\i}a-Ramos, see \cite{R}, see also \cite{HLTXY}. We begin by defining the notion of $\mu$-mean equicontinuity.

\smallskip

\begin{defn} [$\mu$-Mean Equicontinuity] Let $(X,T,\mu)$ be a compact, metric dynamical system. Let $T$
be amenable with a given F\o{}lner sequence $\mathcal{F} = \{F_n\}$.
\begin{enumerate}[label=(\arabic*)]	
\item Let $K\subset X$. A vector $f\in L^2(X,\mu )$ is called a $\mu$-mean equicontinuous vector
on $K$ if given any $\eps >0$ there exists a $\delta > 0$ such that if $x,y\in K$ and $d(x,y) < \delta$ then
$||\psi_{x,f}-\psi_{y,f}||_{B_1} < \eps$.

\item A vector $f\in L^2(X,\mu )$ is called a $\mu$-mean equicontinuous if given any $\eta > 0$,
there exists a compact set $K\subset X$ such that $\mu (K) > 1 - \eta$ and $f$ is $\mu$-mean
equicontinuous on $K$.

\item A dynamical system $(X,T,\mu)$ is $\mu$-mean equicontinuous if each $f\in L^2(X,\mu )$ is a
$\mu$-mean equicontinuous vector.
\end{enumerate}

\end{defn}

\smallskip

\begin{prop} \label{disc-impies-muequi} Let $(X,T,\mu )$ be a compact metric, ergodic dynamical system with $T$ amenable. Let $f\in L^2(X,\mu )$
be a compact vector. Then $f$ is a $\mu$-mean equicontinuous vector. 
\end{prop}
\begin{proof} Since $f$ is $\mu$-compact, it has a representation $f(x) = \sum\limits_{i\in \N}P_if$, where $P_i$ is
the orthogonal projection operator on to subspace $V_i$, (we shall use the previous notation in this proof). Thus given $\eps>0$
we can select $k\in \N$ such that $\big\|f - P\big\|_2< \frac {\eps}{3}$, where $P = \sum\limits_{i=1}^{k}P_if$. Let
$M_1\subset X$ be the of set points at which the ergodic average of $f - P$ converges to $||f-P||_1$. Thus $\mu (M_1) = 1$
and if $x\in M_1$,
$$
\frac {\eps}{3} \geq \big\|f - P\big\|_2 \geq \big\|f- P\big\|_1 = \big\|\psi_{x,f} - \psi_{x,P}\big\|_{B_1}\,.
$$
Thus if $x,y\in M_1$, then
\begin{align}
\big\|\psi_{x,f} - \psi_{y,f}\big\|_{B_1} & \leq \big\|\psi_{x,f} - P(x)\big\|_{B_1} + \big\|P(x) - P(y)\big\|_{B_1} + \big\|P(y) - \psi_{y,f}\big\|_{B_1} \notag \\
& \leq \frac {2\eps}{3} + \big\|P(x) - P(y)\big\|_{B_1}\,. \notag
\end{align}
We show that $\big\|P(x) - P(y)\big\|_{B_1}<\frac {\eps}{3}$, if $x$ and $y$ are close enough. Recall that $P$ has the form
$$
P(x) = \sum\limits_{i=1}^{k}\sum\limits_{j=1}^{d_i}\langle f^i_j(x)\,,w^i_j\rangle w^i_j\,,
$$
where $w^i_j\in V_i\subset L^2(X,\mu )$ and $f^i_j : X \to {\Cc}^{d_i}$ are measurable. Given $\eps > 0$, by Egorov's theorem
pick a compact set $M_2\subset X$ such that $f^i_j$ and $w^i_j$ are continuous on $M_2$. Let $K \subset M_1\cap M_2$, be compact
such that $\mu (K) > 1 -\eps$. Select $\delta > 0$ such that
$$
\text {if}\ d(x,y)< \delta\,,\ x,y\in K\,,\ \text {then}\ \sum\limits_{i=1}^kd_i||f^i_j(x) - f^i_j(y)||\,\big\|w^i_j\big\| < \frac {\eps}{3}\,.
$$
Now for $x,y\in K$, with $d(x,y) < \delta$, the following pointwise representation for $U_tP$ gives
\begin{align}
\big|(U_tP)(x) - (U_tP)(y)\big| & = \big|\sum\limits_{i=1}^{k}\sum\limits_{j=1}^{d_i}\langle (U_tf^i_j)(x) - U_tf^i_j(y))\,,w^i_j\rangle w^i_j\big|\,, \notag \\
& \leq \big|\sum\limits_{i=1}^{k}\sum\limits_{j=1}^{d_i}||U_tf^i_j(x) - U_tf^i_j(y)||\,||w^i_j|| \notag \\
& \leq \sum\limits_{i=1}^kd_i||f^i_j(x) - f^i_j(y)||\,\big\|w^i_j\big\| < \frac {\eps}{3}\,.
\end{align}
Thus,
$$
||P(x) - P(y)||_{B_1} = \limsup\limits_{n \to \infty}\frac {1}{|F_n|}\sum\limits_{t\in F_n}\big|(U_tP)(x) - (U_tP)(y)\big| < \frac {\eps}{3}\,.
$$
Whence, if $x,y\in K$ and $d(x,y) < \delta$, then $\big\|f(x) - f(y)\big\|_{B_1} < \eps$ and the proof is complete.
\end{proof}

\smallskip

For abelian acting groups $T$, the converse of the above theorem is true and there are several proofs, using different arguments,
(see \cite{R}, \cite {HLTXY}). We shall present a result with yet another argument and weaker assumptions, which in particular will yield the converse.
In the following theorem we weaken the `condition of continuity' in the notion of mean equicontinuity to obtain a sufficiency condition for
discrete spectrum. All we need is just one point having three key properties. As for the converse of the above theorem for non-abelian acting groups
none of these proofs will generalize in a straightforward way.

\smallskip

\begin{thm} \label{abelianT} Let $(X,T,\mu )$ be a compact, metric ergodic dynamical system with $T$ abelian.
Let $f\in L^2(X,\mu )$. Suppose there exists a point $x_0\in X$ satisfying the following conditions:
\begin{enumerate}[label=(\arabic*)]	
\item  $x_0$ is $(f,\mu )$ generic, i.e. the ergodic average of $f$ converges at $x_0$ to $\displaystyle \int_X fd\mu$,
\item  $x_0$ is a point of continuity of the map $x\to \psi_{x,f}$, i,e, given $\eps > 0$ there exists a $\delta > 0$ such that
if $d(x,y)< \delta$, then $||\psi_{x,f} - \psi_{y,f}||_{B_1} < \eps$.
\item  For $\delta> 0$, let $R_{x_0}(\delta ) = \{t\in T\ |\ d(x_0,x_0t) < \delta \}$. Suppose for any $\delta > 0$ there exists
a minimal ideal $M \equiv M_\delta \subset \beta T$ such that $\overline {R_{x_0}(\delta )R_{x_0}(\delta )^{-1}}\cap M \ne \emptyset$,
here the closure is in the topology on $\beta T$.
\end{enumerate}
Then $f$ is a $\mu$-compact vector.
\end{thm}
\begin{proof} First, we claim that hypothesis (3) above, implies that there exists a unitary operator $V\in I_\mu\subset \Omega_\mu$
such that  given any $n\in \N$ and $g,h\in L^2(X,\mu )$, there exists $t,s\in R_{x_0}(\frac {1}{n})$ such that
$$
\big|\langle U_{ts^{-1}}g\,,h\rangle - \langle Vg\,,h\rangle \big| < \frac {1}{n}\,.
$$
To prove the claim consider, $F_n = \overline {\{U_{ts^{-1}}\ |\ t,s\in R_{x_0}(\frac {1}{n})\}}\cap I_\mu$, where the closure is in
the topology on $\Omega_\mu$, i.e. in the weak operator topology. Since the family of non-empty closed sets $\{F_n\}$ has the
finite intersection property and $I_\mu$ is compact, $\bigcap\limits_{n\in \N}F_n \ne \emptyset$. Pick any $V\in \bigcap\limits_{n\in \N}F_n$.
The claim follows from this.

Next, let $\eps > 0$ and $h\in L^\infty (X,\mu ) \subset L^2(X,\mu )$ with $||h||_\infty \leq 1$ be given. Using hypothesis (2)
select $n\in \N$ such that $\frac {1}{n} < \frac {\eps}{3}$ and
$$
\text {if}\ d(x_0,y)< \frac {1}{n}\,,\quad \text {then}\ ||\psi_{x_0,f} - \psi_{y,f}||_{B_1} < \frac {\eps}{3}\,.
$$
For this $n$ and taking $g = f$ in the above claim, select $t,s\in R_{x_0}(\frac {1}{n})$ such that
$$
\big|\langle f_{ts^{-1}}\,,h\rangle - \langle Vf\,,h\rangle \big| < \frac {1}{n} < \frac {\eps}{3}\,.
$$
Now, since $T$ is abelian, $\psi_{xt,f}(s) = f(xts) = f(xst) = f_t(xs) = \psi_{x_0,f_t}(s)$, for any $x\in X$, $t,s\in T$. Thus,
if $t\in R_{x_0}\big(\frac {1}{n}\big)$. Then, by our choice of $n$, we have
$$
\frac {\eps}{3} \geq ||\psi_{x_0,f} - \psi_{x_0t,f}||_{B_1} = ||\psi_{x_0,f} - \psi_{x_o,f_t}||_{B_1} = ||\psi_{x_0,(f-f_t)}||_{B_1} = ||f-f_t||_1\,.
$$
The last equality comes from hypothesis (1) and the fact that if $(x_0,f)$ is $\mu$ generic then so is $(x_0t,f)$. Similarly
$||f-f_s||_1 < \frac {\eps}{3}$, if $s\in R_{x_0}(\frac {1}{n})$. Thus $||f-f_{ts^{-1}}||_1 \leq ||f-f_t||_1 + ||f-f_s||_1 < \frac {2\eps}{3}$,
if $t,s\in R_{x_0}(\frac {1}{n})$. Now, note that
\begin{align}
\big|\langle f,h\rangle - \langle Vf,h\rangle\big| & \leq \big|\langle f,h\rangle - \langle f_{ts^{-1}},h\rangle \big| + 
\big|\langle f_{ts^{-1}},h\rangle - \langle Vf,h\rangle\big| \notag \\
& \leq ||f-f_{ts^{-1}}||_1||h||_\infty + \big|\langle f_{ts^{-1}},h\rangle - \langle Vf,h\rangle\big| < \eps \,. \notag 
\end{align}
Since $\eps > 0$ is arbitrary, $\langle f-Vf,h\rangle = 0$ and since $h$, ($||h||_\infty \leq 1$), is arbitrary, $f = Vf$. Since
$V\in I_\mu$, $f$ is $\mu$-compact.
\end{proof}

\smallskip

As a consequence of the above proof, for abelian $T$, we shall give yet another proof of the converse of Proposition \ref{disc-impies-muequi}.
But, first,  we state a result attributed to V. Bergelson, (see \cite {P}), that we shall need here, as well as later. 

\smallskip

\begin{lem} \label {rec-lem} Let $(X,T)$ be a dynamical system, $T$ is countable and $\mu$ be an ergodic Borel probability measure on $X$.
Let $A\subset X$ be a Borel set with $\mu (A) >0$. Let $R \equiv R(x,A) = \{t\in T\ |\ xt\in A\}$. Then
$RR^{-1}$ is a $\Delta^*$ set for any  $x\in Supp (\mu)$. In particular, $RR^{-1}$ is syndetic and hence its closure
in $\beta T$ intersects every minimal right ideal of $\beta T$.
\end{lem}

We recall that a set $A\subset T$ is $\Delta^*$ if and only if it intersects every difference set, i.e. given any infinite
sequence $\{t_n\}$ in $T$, $A\cap \{t_nt_m^{-1}\ |\ n>m\} \ne \emptyset$.

\smallskip

\begin{prop} Let $(X,T,\mu)$ be a compact metric ergodic dynamical system with $T$ abelian and let $f\in L^2(X,\mu )$ be
$\mu$-mean equicontinuous. Then $f$ is $\mu$-compact.
\end{prop}
\begin{proof} Let $\eta >0$ be given. Let $M_c\subset X$ be a Borel set such that $\mu (M_c) > 1-\eta$ and the restriction
to $M_c$ of the map $x\to \psi_{x,f}$ is continuous. Consider sets,
\begin{align}
M_e & = \Big\{x\in X\ |\ \text {ergodic average of}\ f\ \text {converges at}\ x\Big\}\,, \notag \\
M_s & = Supp (\mu )\,, \notag
\end{align}
By the ergodic theorem $\mu (M_e) = 1$ and $\mu (M_s) = 1$ always holds. Pick $x_0\in M_c\cap M_e\cap M_s$. Since $x_0\in M_e$,
it satisfies hypothesis (1) of Theorem \ref{abelianT}. The proof is exactly as the proof of Theorem \ref{abelianT}, except
that the set $R \equiv R_{x_0}(\delta )$ will be replaced by the set $R_{x_0}(\delta )\cap M_c$. Note that since
$x_0\in Supp (\mu )$ and $\mu (R) > 0$, above lemma can be applied to conclude that the closure of the set $RR^{-1}$ in $\beta T$
intersects every minimal right ideal of $\beta T$. Thus hypothesis (3) and `modified hypothesis' (2) of Theorem \ref{abelianT} holds.
So the proof follows exactly as before by selecting $t,s\in R_{x_0}(\frac {1}{n})\cap M_c$.
\end{proof}

\smallskip

\begin{rem}$\quad$

\begin{enumerate}[label=(\arabic*)]	

\item Hypothesis (3) of Theorem \ref {abelianT} demands that the return time set $R\equiv R_{x_0}(\delta)$ satisfy that
$RR^{-1}$ is a `large' set. In fact $R$ may be of density zero. An example of such a case is a set that is piecewise syndetic. In act,
this hypothesis is much weaker, we only need $RR^{-1}$ to be piecewise syndetic. Thus our theorem will yield stronger corollaries
than the one above but we leave this to the reader.

\item Recall that in the definition of $\mu$-mean equicontinuity of a vector $f$, given $\eps > 0$, we
have a compact set $K$ with large measure and a set $S$ with large density so that if $x$ and $y$ in $K$
are sufficiently close, then $f(xs)$ and $f(ys)$ are within $\eps$ for $s\in S$. This is much weaker than demanding
equicontinuity of the family $\{f_s\ |\ s\in S\}$ on $K$. On the other hand if we demand equicontinuity of
this family, where the set $S$ may even have zero density, we can still prove $\mu$-compactness of $f$ if
we assume that $S$ is `large' but not in the sense of density. The following proposition can be proved by similar arguments
to those in Theorem \ref{abelianT}, however, here $T$ can be any (infinite) group. We shall leave the proof to the reader.
\end{enumerate}
\end{rem}

\smallskip

\begin{thm} Let $(X,T,\mu)$ be a compact, metric ergodic dynamical system. Let $f\in L^2(X,\mu )$. Suppose for any
$\eps > 0$ there exists (i) a compact $K\subset X$ such that $\mu(K) > 1- \eps$, (ii) a minimal ideal $M\in \beta T\backslash T$
and (iii) a set $S\subset T$ such that the following holds.
\begin{enumerate}[label=(\arabic*)]
\item $\{f_s\cdot\chi_K\ |\ s\in S\}$ is an equicontinuous family and
\item  for any infinite sequence $\{s_n\}$ in $S$, the difference set $D(S) \equiv \overline {\{s_ns_m^{-1}\ |\ n>m\}}\cap M \ne \emptyset$,
(where the closure is in $\beta T$).
\end{enumerate}
Then $f$ is a $\mu$-compact vector.
\end{thm}

\smallskip

\section{A Characterization of $\mu$-compact vectors using the notion of `witness of irregularity'}

\medskip

Following ideas of M. Talagrand and others, we shall see another characterization of $\mu$-compact vectors using the notion
of a `Glivenko-Cantelli family', (see Theorem (\ref{Spectral-tame})). In probability theory one studies `Glivenko-Cantelli family'
of random variables to investigate conditions on the family under which the law of large numbers has a `uniform version'. In
the later half of this section we compare `weakly tame' with tame systems, (Theorem (\ref{GGM})). It is our hope that this
exposition and proofs will make contributions in references \cite{Fre}, \cite{GM2},\cite{GM4}, \cite{T1},\cite{T2} more accessible
from the dynamics point of view. First, we begin by recalling the notion of a Glivenko-Cantelli family.

\smallskip

\begin{defn} Let $(X,\mathcal{B},\mu )$ be a complete probability space. Let $\mathcal{L}^1 \equiv  \mathcal{L}^1(X,\mu )$ denote the space of
	measurable functions $f$ such that $\bE(f) \equaldef \int_Xfd\mu < \infty$. We shall not identify functions in $\mathcal{L}^1$ with their classes
	in $L^1(X,\mu )$.
	\begin{enumerate}
		\item [(1)]A subset $Z\subset \mathcal{L}^1$ is ordered bounded if there exists a $u\in \mathcal{L}^1$, $u\geq 0$ such that for each $f\in Z$,
		we have $|f(x)| \leq u(x)$, $\forall x \in X$.
		
		\item [(2)]We also recall the following notion of uniform ergodicity \`a la Glivenko-Cantelli. A subset or a
		family $Z\subset \mathcal{L}^1$ is $\mu$-Glivenko-Cantelli if $\lambda$-a.e.
		$$
		\lim\limits_{N\to \infty}\sup_{f\in \F}\Bigg|\frac{1}{N}\sum\limits_{k=1}^{N}f(\omega_i)-\int fd\mu \Bigg| = 0\,.
		$$
		Here $\omega = (\omega_0, \omega_1, \omega_2, \cdots )$ is an i.i.d. process with common distribution $\mu$;
	\end{enumerate}
\end{defn}

\smallskip

First let $T = \Z$. We shall now show that when our underlying probability space is a compact metric dynamical systems $(X,T,\mu )$, for $f\in L^2(X,\mu )$,
the $\mu$-compactness of $f$ is equivalent to the family $Z \equiv \mathcal{F} = \{f\circ T^n\ |\ n\in \Z\}$ being $\mu$-Glivenko-Cantelli. We begin with one
of the implication.

\smallskip

\begin{prop} Suppose $(X,T,\mu )$ be a compact, metric ergodic dynamical system. Let $f\in L^2(X,\mu )$ be a $\mu$ compact vector. Then the family
	$\{f\circ T^n\ |\ n\in \Z\}$ is a $\mu$-Glivenko-Cantelli family.
\end{prop}
\begin{proof} First we verify this when $f$ is an $L^2$ eigenfunction, i.e. $f(Tx) = e^{2\pi i\lambda}f(x)$, a.e $x$. By ergodicity, we can assume that
	$f$ is bounded almost surely. We need to verify that the family $\{f\circ T^n\ |\ n\in \Z\}$ is $\mu$-Glivenko-Cantelli. Note that
	\begin{align}
	\sup_{m}\Bigg|\frac{1}{N}\sum_{k=1}^{N}f\circ T^m(\omega_k)\Bigg| & = \sup_{m}\Bigg|\lambda^m\frac{1}{N}\sum_{k=1}^{N}f(\omega_k)\Bigg| \notag \\
	& = \frac{1}{N}\Bigg|\sum_{k=1}^{N}f(\omega_k)\Bigg|\,.
	\end{align}
	The last term converge to zero by the law of large numbers. Now, let $f\in L^2(X,\mu )$ be $\mu$-compact. Then we can write
	$f=\sum_{j=0}^{+\infty}\alpha_j f_j$, where each $f_j$ is an eigenfunction and $\sum_{j=0}^{+\infty}|\alpha_j|^2<\infty$ and the rest is clear.
\end{proof}

\smallskip

Now, we prove the converse.

\smallskip

\begin{prop} Suppose $(X,T,\mu )$ be a compact, metric ergodic dynamical system. Suppose the family
	$\{f\circ T^n\ |\ n\in \Z\}$ is a $\mu$-Glivenko-Cantelli family.. Then $f\in L^2(X,\mu )$ is a $\mu$ compact vector.
\end{prop}
\begin{proof}
	First write the `compact-weak-mixing decomposition' of $f$, $f = f_c + f_{wm}$, where $f_c$ is $\mu$-compact and $f_{wm}$ is a non-constant 
	$\mu$-weak mixing vectors in $L^2(X,\mu )$. From the above proposition we know that $\{f_c\circ T^n\ |\ n\in \Z\}$ is $\mu$- Glivenko-Cantelli.
	Thus, the hypothesis implies that $\{(f-f_c)\circ T^n\ |\ n\in \Z\}$ is $\mu$-Glivenko-Cantelli, i.e. the family $\{f_{wm}\circ T^n\ |\ n\in \Z\}$
	is $\mu$-Glivenko-Cantelli. Since $f_{wm}$ is weakly-mixing, along a subsequence $(n_j)$ of density $1$, we have,
	$$
	\lim\limits_{j\to \infty}\langle f_{wm}\circ T^{n_j},g\rangle = \int  f_{wm} d\mu \int g d\mu \,.\quad \text {for any}\ g \in L^{\infty}(X,\mu)\,.
	$$
	Since the family $\{f_{wm} \circ T^n, n \in \Z\big\}$ is $\mu$-Glivenko-Cantelli,  by a lemma of M. Talagrand, (\cite[Proof of
	Proposition 2.5, p. 379]{T2}, or \cite[Chap. 46, pp. 59-60]{Fre} ), there exist a finite sub-algebra $\mathcal{P}$ of the Borel sigma algebra of $X$
	such that, for any $j\in \N$, one has
	\begin{equation} \label{approx}
	\big\|f_{wm} \circ T^{n_j} - {\bf E}(f_{wm} \circ T^{n_j}|\mathcal{P})\big\|_1 < \eps\,,
	\end{equation}
	where ${\bf E}(\cdot |\mathcal{P})$ is the projection operator of the conditional expectation with respect to $\mathcal{P}$. Moreover,
	by the property of this projection operator  we have,
	\begin{equation} 
	{\bf E}(f_{wm} \circ T^{n_j}|\mathcal{P}) = \sum_{P \in \mathcal{P}} \frac{1}{\mu(P)}\Big(\int_P f_{wm} \circ T^{n_j}d\mu\Big) \chi_P(x)\,\,.
	\end{equation}
	Letting $j\to \infty$, we get,
	\begin{equation} 
	\lim\limits_{j\to \infty}{\bf E}(f_{wm} \circ T^{n_j}|\mathcal{P}) = \int f_{wm} d\mu \sum_{P \in \mathcal{P}}  \chi_P(x)\,,
	\end{equation}
	for almost all $x \in X$. This combined with \ref{approx} yields
	\begin{equation}
	\limsup_{j \to \infty }\big\|f_{wm} \circ T^{n_j}- \sum_{P \in \mathcal{P}}  \chi_P(x) \int f_{wm} d\mu \big\|_1 \leq \eps \,.
	\end{equation}
	Whence
	\begin{equation}
	\Big\|f_{wm}-\int f_{wm} d\mu\Big\|_1 \leq \eps \,,,
	\end{equation}
	which is impossible since, $f_{wm}$ is a non-constant weak-mixing vector. The proof of the proposition is complete.
\end{proof}

\smallskip

\begin{rem} Now we try to capture $\mu$ compactness of the vector $f$ in terms of `up-crossings'. The notion of `witness of irregularity'
	captures the case when the `up-crossings of the family $\mathcal{F} \equiv \{f_t\ |\ t\in T\}$ are `wild' and $\mu$-compactness of $f$
	can be characterized when this does not happen. Following M. Talagrand, two non-negative numbers are introduced which are in some sense
	another notions of `entropy' for the above family $\mathcal{F}$. Again, $\mu$-compactness of $f$ can be characterized by vanishing of
	these two numbers and in this sense they `better' serve the purpose of characterizing $\mu$-compactness of $f$ than the usual measure theoretic
	or topological entropy. These notions will also allow us to generalize above two proposition to the setting of ergodic actions of countable,
	amenable groups. For the following concepts we refer to \cite{T2}.
\end{rem}

\smallskip

\begin{defn} We shall first recall our more general set up, where $(X,\mathcal{B},\mu )$ be a complete probability space. Let $Z$ be a uniformly
	bounded subset of real valued measurable maps on $X$. First we introduce quantities $N^1(Z,n,\eps)$ and $N^{\infty}(Z,n,\eps)$ whose
	logarithmic growth rate, (as $n\to \infty$), will capture a version of the `notion of entropy' for the family $Z$.
	
	\noindent (1) For a fixed $n\in \N$, consider the following norms on ${\R}^n$,
	$$
	\big\|x\big\|_1 = \frac {1}{n}\sum\limits_{k=1}^{n}x_k\,,\quad \text {and}\quad ||x||_\infty = \max \{|x_k|\ |\ 1\leq k\leq n\}\,.
	$$
	Let $X_k:\Omega \to X$, $(k\in \N)$  be an $X$ valued i.i.d. whose common distribution is $\mu$. Consider the random set
	$$
	\Sigma_Z(\omega ,n) \equiv \{(f\circ X_1(\omega ),\cdots ,f\circ X_n)(\omega ))\ |\ f\in Z\}\subset {\R}^n\,.
	$$
	Let $N^1(Z,n,\eps)(\omega )$ and $N^{\infty}(Z,n,\eps)(\omega)$ denote the minimum number of balls of radius $\eps$, (in the respective metric on
	${\R}^n$), required to cover the set $\Sigma_Z(\omega ,n)$.\\
	
	\noindent (2) Next, we recall the notion of `shattering'. As before $Z$ is a family of measurable, real valued maps on $X$. Let $\alpha < \beta$
	be two real numbers. A finite subset $F\subset X$ is said to be $(\alpha,\beta )$-shattered by $Z$ if given any subset $G\subset F$, there is some
	$f\in Z$ such that
	\begin{align}
	x\in G\quad & \text {implies}\quad f(x) < \alpha \,,\text {and} \notag \\
	x\in F \backslash G \quad & \text {implies}\quad f(x) > \beta \,. \notag
	\end{align}
	
	\noindent (3) Now to avoid measure theoretic technicalities, we shall suppose the family $Z$ is countable. Then it follows that for each fixed $n\in \N$
	the following subset $S^{(\alpha ,\beta )}_n$ of $X^n$ given by
	$$
	S^{(\alpha ,\beta )}_n = \{(x_1,\cdots ,x_n)\in X^n\ |\ \text {the set}\ F = \{x_1, \cdots ,x_n\}\ \text {has}\ n\ \text {elements and is}
	\ (\alpha ,\beta)\ \text {shattered by}\ Z\}\,
	$$
	is a measurable subset, since $Z$ is countable. The largest $n\in \N$ for which $\mu ^n\big(S^{(\alpha ,\beta )}_n\big) > 0$
	will be called the $(\alpha ,\beta ,\mu)$-shattering dimension of the family $Z$.
	
	\noindent (4) A measurable set $A\subset X$ is a $(\alpha ,\beta )$-witness of irregularity, (of the family $Z$), if (i) the restriction of $\mu$ to $A$
	is non-atomic and (ii) for each $n\in \N$, almost all $(x_1,\cdots,x_n)\in A^n$ belong to $S^{(\alpha ,\beta )}_n$.
	
	\noindent (5) A countable family $Z$ is said to have no witness of irregularity, (with respect to $\mu$), if any measurable set $A\subset X$ that is
	a $(\alpha ,\beta )$ witness of irregularity for some $\alpha < \beta$, then $\mu (A) = 0$.
\end{defn}

\smallskip

Then the following is a summary of results proved in \cite{T2}.

\smallskip

\begin{thm} \label{Talagrand} Let $(X,\mathcal{B},\mu )$ be a complete probability space. Let $Z$ be a uniformly bounded countable subset of real
	valued measurable maps on $X$. Then the following statements are equivalent.
	\begin{itemize}
		\item [(1)] The family $Z$ is a Glivenko-Cantelli family,
		\item [(2)] $\lim\limits_{n\to \infty}\frac {1}{n}\E\big(log (N^1(Z,n,\eps ))\big) = 0$, for each $\eps > 0$,
		\item [(3)] $\lim\limits_{n\to \infty}\frac {1}{n}\E\big(log (N^\infty(Z,n,\eps ))\big) = 0$, for each $\eps > 0$,
		\item [(4)] family $Z$ has no witness of irregularity with respect to $\mu$.
		\item [(5)] for any $\alpha < \beta$, $\lim\limits_{n\to \infty}\big(\mu^n(S^{(\alpha ,\beta )}_n\big)^{1/n} = 0$.
	\end{itemize}
	Here $\E$ denotes the expectation of the random variables on $X^n$.
\end{thm}

\smallskip

Given a compact metric topological dynamical system $(X,T,\mu )$ with $T$ countable, amenable and $f\in C(X)$, we can apply above theorem to 
the family $Z = \{f_t\ |\ t\in T\}$ to get the following characterization of a $\mu$-compact vector $f$.

\smallskip

\begin{prop} Let $(X,T,\mu )$ be a compact, metric ergodic dynamical system with $T$ countable, amenable. Let $f\in C(X)$, then the following statements are equivalent.
	\begin{itemize}
		\item [(1)] $f$ is a $\mu$-compact vector,
		\item [(2)] the family $\mathcal{F} \equiv \{f_t\ |\ t\in T\}$ is Glivenko-Cantelli, where $f_t(x) = f(xt)$, $x\in X$, $t\in T$.
		\item [(3)] $\lim\limits_{n\to \infty}\frac {1}{n}\E\big(log (N^{i}(\mathcal{F},n,\eps ))\big) = 0$, for each $\eps > 0$, where $i = 1$ or $\infty$.
		\item [(4)] $\lim\limits_{n\to \infty}{\mu}^n\big(S^{(\alpha ,\beta )}_n\big)^{1/n} = 0,$ for any $\alpha < \beta$.
		\item [(5)] The family $\mathcal{F}$ has no witness of irregularity.
	\end{itemize}
\end{prop}
\begin{proof} We shall prove that (1) implies (3) and (2) implies (1), the equivalence of (2), (3), (4) and (5) is established by the previous theorem.
	
	\noindent (1) implies (3) : Consider the set $\Sigma_Z(\omega ,n) \equiv \Sigma_f(x_1,\cdots ,x_n,n)\subset {\R}^n$ defined above, (where
	$Z  \{f_t\ |\ t\in T\}$ and $\omega = (x_1,\cdots ,x_n)\in X^n$). We shall show that for any $\eps >0$, $N^{\infty}(Z,n,\eps)(\omega)$ is bounded independent of $n$,
	for all $\omega = (x_1,\cdots ,x_n)\in X^n$ outside a set of arbitrarily small $\mu^n$ measure. This will prove (3). Now recall from Proposition 3.1, on a set $K$
	of measure arbitrarily close to $1$, $f(xt)$ can be written as $U_t P(x)$, where  $P(x) = \sum\limits_{i=1}^{k}\sum\limits_{j=1}^{d_i}\langle f^i_j(x),w^i_j\rangle w^i_j$,
	(here we are using the notation of Proposition 3.1). Here each $f^i_j : K\to {\mathbb C}^{d_i}$ is continuous. Now we notice that it is enough to show
	that $N^{\infty}(Z^i_j,n,\eps)(\omega)$ is bounded independent of $n$ for $\omega \in K^n$ for each $i$ and $j$, where Here $Z^i_j = \{(f^i_j)_t\ |\ t\in T\}$.
	Of course this family is a family of ${\R}^{d_i}$ valued maps on $X$ instead of being real valued, but that does not change anything. Now note that
	$f^i_j(x t) = U_t f^i_j(x)$, where $U_t$ is a finite dimensional unitary operator acting on a vector $f^i_j(x) \in {\Cc}^{d_i}$. Given $\epsilon > 0$ select
	a compact neighbourhood $V$ of the identity in the Unitary group $U(d_i)$ so that $||v-U_tv||<\eps$ for all $v\in {\Cc}^{d_i}$. Let $L$ be the
	number of translates of $V$ that cover the whole unitary group $D(d_i)$. Then it follows that the set $\{(U_tf^i_j(x_1),\cdots ,U_tf^i_j(x_n))\ |\ t\in T\}$
	is covered by at most $L$ balls of radius $\eps$, for $(x_1,\cdots ,x_n)\in K^n$. Note that (i) $L$ is independent of $n$ and depends only on $\eps$ and
	(ii) for $\omega$'s outside $K^n$ - a set of arbitrarily small $\mu^n$ measure - the number $N^{\infty}(Z^i_j,n,\eps)(\omega )$ is bounded by
	$\big(\frac {M}{\eps}\big)^n$, where $M$ is the maximum of $f^i_j$. Thus, given $n$ and $\eps$, choosing set $K$ to be a compact set with measure very close to $1$,
	we can make the expectation of the map $\omega\to log (N^{\infty}(Z^i_j,n,\eps)(\omega ))$ bounded independent of $n$. This proves (3).
	
	\noindent (2) implies (1) : This proof is exactly as in the case of $T = \Z$. Because $T$ is countable amenable, if we have a weak-mixing vector
	$f\in L^2(X,\mu)$, then $f_{t_n}\to 0$ weakly along a set $\{t_n\ |\ n\in \N\}$ of density one, (density is with respect to some fixed F\o{}lner sequence).
	Then the rest of the proof is as in the case of integer group.
\end{proof}

\smallskip

\begin{rem} It is easy to see that if for any $\alpha < \beta$, the $(\alpha ,\beta )$-shattering dimension is finite, then $f$
	is $\mu$-compact. The converse may not hold, but the above condition on the limit in (4) guarantees the converse.
\end{rem}

\smallskip

Now we introduce the notion of `weak-tameness'.

\smallskip

\begin{defn} A compact, metric dynamical system $(X,T)$ is weakly tame if each $f\in C(X)$ is $\mu$-tame for and every
	invariant Borel probability measure $\mu$ on $X$.
\end{defn}

\smallskip

The following characterization already follows from the results proved earlier.

\smallskip

\begin{thm}\label{Spectral-tame} Let $(X,T)$ be a compact, metric system where $T$ is countable and amenable. Then the following
	statements are equivalent.
	\begin{itemize}
		\item [(1)] $(X,T)$ is weakly tame,
		\item [(2)] $(X,T,\mu )$ has discrete spectrum with respect to every invariant Borel probability $\mu$ on $X$,
		\item [(3)] The family $\{f_t\ |\ t\in T\}$ is $\mu$-Glivenko-Cantelli for every $f\in C(X)$ and every invariant Borel probability
		measure $\mu$ on $X$.
		\item [(4)] For every $f\in C(X)$, the family $\{t_t\ |\ t\in T\}$ has no witness of irregularity with respect to $\mu$ for every
		invariant Borel measure $\mu$ on $X$.
	\end{itemize}
\end{thm}

\smallskip

\begin{rem} We would like to pose a question : Can we give a characterization of a subclass of weakly tame dynamical systems
	for which all for which for any $\alpha < \beta$, the $(\alpha ,\beta )$-shattering dimension is finite and bounded as a function
	of the invariant measure $\mu$?
\end{rem}

\smallskip

In the following we shall compare `tameness' with `weak tameness'. The above characterization of weak-tameness is based on the
up-crossing behavior of the family $\mathcal{F} \equiv \{f_t\ |\ t\in T\}$. In topological dynamics one tries to capture
dynamical features in terms of the regularity properties of functions in the family $\{f_p\ |\ p\in \beta T\}$, which is the
closure of the family $\mathcal{F}$ in the topology of point-wise convergence. Recall that $(X,T)$ is tame if and only if each
$f_p$ is a Baire-$1$ function, for each $f\in C(X)$ and $p\in \beta T$. Thus, in some sense tameness can be studied by a
`topological version' of up-crossings. A more formal way to study this is by introducing the notion of `independence' for a
sequence of a pair disjoint subsets of $X$. But before introducing this notion, first we shall state a characterization of
tameness and ask the reader to compare it with the previous theorem.

\smallskip

\begin{thm} \label{GGM}Let $(X,T,\mu)$ be a compact metric topological dynamical system with $T$ countable. Then the following
	statements are equivalent;
	\begin{itemize}
		\item [(1)] $(X,T)$ is tame,
		\item [(2)] The family $\{f_t\ |\ t\in T\}$ is $\mu$-Glivenko-Cantelli for every $f\in C(X)$ and every Borel probability measure
		$\mu$ on $X$.
		\item [(3)] For every $f\in C(X)$, the family $\{f_t\ |\ t\in T\}$ has no witness of regularity with respect to $\mu$ for every
		Borel probability measure $\mu$ on $X$.
	\end{itemize}
\end{thm}

\smallskip

Before getting into the technicalities of this proof we make some observations. We shall mention two theorems whose proofs share
a certain `common part' with the proof of the above theorem. The first one is due to D. Kerr and H. Li (see Proposition 6.4 of
\cite {Kerr-Li}).

\smallskip

\begin{thm}\label{KL} [Kerr-Li Theorem] A system $(X,T)$ is not tame if and only if there is a non-trivial `IT pair',
	(this notion will be defined shortly).
\end{thm}

\smallskip

The second one is the famous dichotomy theorem of J. Bourgain, D. Fremlin and M. Talagrand. The following is its
`dynamical formulation' due to E. Glasner.

\smallskip

\begin{thm}\label{BFT} [BFT-dichotomy Theorem] A compact, metric dynamical system $(X,T)$, (with $T$ countable),
	is not tame if and only if the enveloping semi-group $E(X,T)$ contains a topological copy of $\beta T$.
\end{thm}

\smallskip

As mentioned before, the argument in proofs of the last three  theorems have a certain `common part'. This is
the part that involves `combinatorics', namely a `dichotomy argument' using the `Nash-Williams Theorem'. In the following
we shall show how the arguments in the above theorems can be sketched in a self contained way using only section 5 of
S. Todorcevic's book \cite {To} and a lemma from \cite {RVH}. We begin by recalling the notion of `independence'.

\smallskip

\begin{defn} ($1_a$) A sequence $\{({A_0}^n,{A_1}^n)\}_{n\in \N}$ of disjoint pairs of subsets of a set $X$ is
	independent if for every finite subset $F\subset \N$, and $G\subset F$, we have
	$$
	\bigcap_{j\in G}A_0^{j} \bigcap \bigcap_{j\in F \setminus G}A_1^{j} \neq \emptyset\,.
	$$
	
	\noindent ($1_b$) The same sequence is $\sigma$-independent if for every subset $F\subset \N$ we have,
	$$
	\bigcap_{j\in F}A^j_0 \bigcap \bigcap_{j\not \in F}A^j_1 \neq \emptyset\,.
	$$
	
	\noindent (2) Let $\mathcal{F} = \{f_n\ |\ n\in \N\}$ be a countable family of real valued functions on a set $X$.
	This family is said to be independent (or $\sigma$-independent) at level $(\alpha , \beta)$, (where $\alpha < \beta$
	are reals), if the sequence of sets $A^n_0 = f_n^{-1}(-\infty ,\alpha]$ and $A^n_1 = f_n^{-1}[\beta ,\infty)$.
	
	\noindent (3) Let $(X,T)$ be a compact, metric dynamical system. A pair $(x,y)\in X\times X$ is said to be an IT-pair if any product
	neighborhood $U\times V$ of $(x,y)$ has an infinite independent set, i.e. there is an infinite countable set $B\subset \T$
	such that the sequence $\{(U_t,V_t)\ |\ t\in B\}$ is an independent sequence, (without loss of generality, we are assuming
	that $U\cap V = \emptyset$).
\end{defn}

\smallskip

\noindent {\bf Proof of Theorem \ref{GGM}}

\smallskip

\begin{proof} First we recall the Rosenthal dichotomy theorem and its proof, (Chapter 5, \cite{To}), (which is a weaker
	form of the BFT theorem). A strong version is needed to prove the BFT theorem, see \cite[Theorem 11]{B76}.
	
	\smallskip
	
	\begin{thm} [Rosenthal Dichotomy Theorem] Let $X$ be a compact metric space and $\{f_n\}$ be a point-wise bounded sequence
		in $C_p(X)$- the space of continuous real valued maps on $X$ with the topology of pointwise convergence. Then either $\{f_n\}$
		contains a convergent sequence, (i.e. as a sequence in $C_p(X)$) or it contains a subsequence whose closure is homeomorphic
		to $\beta \N$.
	\end{thm}
	
	\smallskip
	
	\noindent If we look at the proof of this theorem, (page 22-23 of \cite{To}), it uses the Nash-Williams theorem to conclude
	that if the point-wise closure of $\{f_n \, | \, n \in \N\}$ does not contain a (point-wise) convergent sub-sequence then
	for some $\alpha < \beta$, there is a sub-sequence $\{f_{n_i} \, | n_i\in \N \}$ such that this sub-sequence is actually-$\sigma$
	independent at level $(\alpha ,\beta )$. This fact is what we called `the `common part' to the proofs of all of the three
	theorems mentioned above. In the following we very briefly describe the arguments in the proofs of these theorems.
	
	\noindent(i) After this `common part', to prove the Rosenthal dichotomy theorem (or the BFT theorem) one only needs some facts
	about the topology of a zero dimensional space, (see \cite{To}, Chapter 5 and 13).
	
	\noindent(ii) In the proof of Kerr-Li Theorem, to get a non-trivial IT pair from an independent sub-sequence one has to use
	`more combinatorics' and a compactness argument, (see Proposition 3.9 and 6.4 of \cite {Kerr-Li}). A `Ramsey type' combinatorial
	property one needs is the following : If $U\times V$ has an infinite independent set and $U = U_1\cup U_2$ then either
	$U_1\times V$ or $U_2\times V$ has an infinite independent set.
	
	Now we gp to the proof of Theorem \ref{GGM}. As mentioned in the above remark, either the system $(X,T)$ is tame or for some
	$f\in C(X)$, the family $\mathcal{F}_B = \{f_t\ |\ t\in B\subset T\}$ is a $\sigma$-independent family, where $B$ is a
	(countable) infinite set. Now we use Theorem 1.3 of \cite{RVH} which says that non existence of a $\sigma$-independent infinite
	sub-sequence in the family $\mathcal{F} = \{f_t\ |\ t\in T\}$ is equivalent to saying that the family $\mathcal{F}$ is
	`universally Glivenko-Cantelli' i.e. it is $\mu$-Glivenko-Cantelli for every Borel probability measure $\mu$ on $X$.
	
	So the dichotomy argument in the proof of Rosenthal's theorem followed by the above theorem of van Handel gives a proof of
	Theorem \ref {GGM}. It is worth noting that existence of a $\sigma$-independent sub-sequence $\mathcal{F}_B$ allows one to
	construct a Bore probability measure $\nu$ on $X$ such that the family $\mathcal{F}_B$ is not $\nu$-Glivenko compact.
	This construction is Theorem A.1 in the appendix of \cite {RVH} and should be of independent interest despite work
	summarized in Theorem \ref{Talagrand}.   
\end{proof}

\smallskip

\begin{rem} For $T = \Z$, the statement of Theorem (\ref {GGM}) is motivated by Theorem 8.20 and Theorem 8.16 of \cite {GM2}.
	However, the authors have neither given a proof or a hint to Theorem 8.20, nor do they give any reference to Theorem 8.16.,
	(they attribute this theorem to M. Talagrand). In this paper the authors define notion of `topological stability' of the
	family $\mathcal{F}$ by requiring that such family do not have any `topologically critical' set. The authors define a closed
	set $A\subset X$ `topologically critical' if for some $\alpha < \beta$ the set $S^{(\alpha ,\beta )}_n$ is dense in $A^n$
	for every $n\in \N$. A more natural definition of `topologically critical' would have been given by demanding that ${\mu}^n$
	almost every point of $A^n$ belongs to $S^{(\alpha ,\beta )}_n$ for every $n$ and every probability measure $\mu$. These two
	notions may not be the same. In any case these notions are trying to capture the topological analogue of $(\alpha ,\beta )$-witness
	of irregularity.
\end{rem}

\smallskip

Finally we state a couple of consequences of various things discussed above.

\smallskip

\begin{cor} A dynamical system is weakly tame if and only if has a bounded measure theoretic complexity with respect to any invariant
	measure.
\end{cor}
\begin{proof} By Proposition 4.1 from \cite{HWY}, if $\mu$ is an invariant measure and the dynamical system $(X,\mathcal{A},\mu,T)$
	has a discrete spectrum, then its measure complexity is bounded. Conversely, by the main result of \cite{HLTXY} ([Theorem 4.3]) or
	\cite[Theorem 3.2]{Vershik}, if the measure complexity is bounded then the system is $\mu$-mean equicontinuous, hence, its spectrum
	is discrete. Therefore the dynamical system has a bounded complexity for any invariant measure if and only if for each invariant
	measure its spectrum is discrete. The proof of the corollary is complete.
\end{proof}

\smallskip

Now let $T = \Z$. We recall the notion of a null system.

\smallskip

\begin{defn} Let $(X,T)$ be a compact metric dynamical system. Given a sequence $S=\{s_i\} \subset \Z$ and a finite open cover
	$\mathcal{O}$ of $X$, we define the topological entropy of $(X,T)$ with respect to $S$ and  $\mathcal{O}$ by
	$$
	h_{\textrm{top}}(S, \mathcal{O})=\lim_{n \rightarrow +\infty }\frac{\log\Big(N\big(\bigvee_{i=1}^{n}T^{s_i}\big( \mathcal{O}\big)\big)\Big)}{n},
	$$
	where $N(.)$ is the minimal cardinality of a subcover. The sequential topological entropy of $T$ along $S$ is given by
	$$
	h_{\textrm{top}}(S)=\sup\Big\{h_{\textrm{top}}(S, \mathcal{U}), \mathcal{U} \textrm{~is~an~ open~ cover~  of~} X \Big\}.
	$$
	System $(X,T)$ is said to be null if its sequential topological entropy is zero for any subsequence. 
\end{defn}

\smallskip

\begin{cor} Suppose $(X,T)$ is null, then $(X,T)$ is weakly tame.
\end{cor}
\begin{proof} By Kushnirenko \cite{Kush}, for null systems any invariant measure has discrete spectrum and hence by our result
	the system is weakly tame.
\end{proof}

\smallskip
\begin{rem} (1) D. Kerr and H. Li have characterized null systems as those that do not admit non-trivial IN-pairs, (see
	\cite {Kerr-Li} for details).
	
	\noindent (2)Very recently, Fuhrmann and Kwietniak proved that there is a tame dynamical system which is non-null \cite{FK} (Tame
	in the sense of Glasner-K\"{o}hler). 
\end{rem}

\medskip

\medskip

\section{\texorpdfstring{The Veech systems and $\K(T)$}{}}

\smallskip

In \cite{V1} Professor W. Veech introduced a structure which he called `a bi-topological flow'. The following is a slight
modification of his original definition.

\smallskip

\begin{defn} Let $(X,T)$ be a topological dynamical system, where $\tau_1$ denotes the topology on $X$. Let $T$ be countable. The
system $(X,T)$ is said to be a Veech system if $X$ has another topology $\tau_2$ such that the following properties hold.
\begin{enumerate}[label=(\arabic*)]
\item  Topology $\tau_2$ is a metric topology generated by a metric $D:X\times X\to [0,\infty )$,
\item  $\tau_1 \leq \tau_2$,
\item  any $\tau_2$ open set is $\tau_1$-Borel, (i.e. is in the sigma algebra generated by the $\tau_1$ open sets).
\item The $T$ action preserves metric $D$, i.e. $D(xt,yt) = D(x,y)$, for all $x,y\in X$ and $t\in T$.
\item  The space $(X,\tau_2)$ is separable.
\end{enumerate}
\end{defn}

\smallskip 

\begin{rem} (1) In fact W. Veech introduced this structure in two of his papers, first in \cite {V1} and much later in \cite {V2}.
In the first paper instead of condition (3) above, he requires a much stronger condition of $\tau_1$-continuity of the map
$y\mapsto D(x,y)$ for a generic set of $x$'s. In his later paper he weakened it by demanding that $y\mapsto D(x,y)$ be lower
semi-continuous. In the second paper, his main interest was in studying the special case of the translation flow on orbit closure of
functions of class $\K(\Z)$ (which we shall recall below). For this system the condition in his first paper does not hold
but the one in the second paper holds. In his study, he posed the question : Whether the Sarnak conjecture\footnote{See sections 6. for more details.} holds for the
translation flow on the orbit closure of functions of class $\K(\Z)$. In a recent paper \cite {HWY} the authors claim to have proved this,
\cite[Theorem 5.1]{HWY}. However this proof has a gap. We shall discuss this and present a correct proof of this conjecture.

\noindent (2) W. Veech introduces this structure for uncountable acting groups $T$ as well. In general, for such groups
technicalities arise due to non-separability of $l^\infty (T)$ and hence even the definition of $\K(T)$ becomes cumbersome.
So we restrict ourselves to countable $T$'s.
\end{rem}

\smallskip

As mentioned above, a prime example of Veech-system is the translation flow on  the orbit closure of a function of `class $\K(T)$'. The precise definition follows.

\smallskip

\begin{defn} Consider $l^\infty (T)$-the space of bounded, complex valued function on $T$ with the weak* topology as a dual of $l^1(T)$.
Let $f\in l^\infty (T)$ and $X_f$ be the closure of the orbit $\{f_t\ |\ t\in T\}$ with respect to the weak* topology,
where $f_t(s) = f(st)$. A function $f\in l^\infty (T)$ belongs to the class $\K(T)$ if $X_f$ is separable with respect to
the topology induced by the restriction of the $l^\infty(T)$ norm to $X$. It is not difficult to verify that the translation flow
$(X_f,T)$ is a Veech system. Here $\tau_1$ is the weak* topology and $\tau_2$ is the $l^\infty (T)$ norm topology on $X_f$.
\end{defn}

\smallskip 

\begin{rem} One can show that if $f\in \K(T)$ then $X_f\subset \K(T)$ and $\K(T)$ is a subalgebra containing the subalgebra $\textrm{WAP}(T)$
of weakly almost periodic functions on $T$. The following is a concrete example when $T = \Z$, that shows that this containment is proper.
\end{rem}

\smallskip

\begin{exmpl} Here $T = \Z$. Let $S\equiv \{n_k\}$ be a sequence in $\N$ such that $n_{k+1}-n_k$ increases to $\infty$. Let
$\bar \eps \equiv \{\eps_k\}\in \{-1\,,1\}$. Corresponding to $(S,\bar \eps)$, define a map $f \equiv f^{(S,\bar \eps)}:\Z \to \{-1,0,1\}$
by setting
\begin{align}
f(n) & = 0\ \text {if}\ n\leq 0\,, \notag \\
& = \eps_k\,, \text {if}\ n_k\leq n<n_{k+1}\,, k\in \N\,. \notag
\end{align}
Consider $f$ to be a point in $\{-1,0,1\}^{\Z}$ and let $X_f$ be the orbit closure of $f$ under the left shift map.
\end{exmpl}

\smallskip

\begin{lem} Consider the above example, then its enveloping semigroup $E(X_f,\Z)$ is given by
$$
E(X_f,\Z) = \Z\,\cup \{\hat p,\hat q,\hat z\}\,
$$
where the elements $\hat p,\hat q$ and $\hat z$ of $E(X_f,\Z)$ will be described in the proof. In particular, $f\in K(\Z)$ and $(X_f,\Z)$
is a Veech system. Furthermore, every element of $E(X_f,\Z)$ is a Baire-1 function and hence $(X_f,\Z)$ is tame. Finally $f\notin \textrm{WAP}(\Z)$ i.e
$(X_f,\Z)$ is not weakly almost periodic.
\end{lem}
\begin{proof} Recall that  $f \equiv f^{(S,\bar \eps)}$ is given. To avoid confusion, we shall denote the point $f$ of $X_f$ by $x^*$. Denote
by $[a_k,b_k]$ the `middle third' of the interval $[n_k,n_{k+1}]$. Let
$$
P = \bigcup \Big\{n\in [a_k,b_k]\ |\ f(a_k) = 1\Big\}\,,\quad Q = \Big\{n\in [a_k,b_k]\ |\ f(a_k) = -1 \Big\}\quad \text {and}\ Z = \Z \backslash (P\cup Q)\,.
$$
The partition $\{P,Q,Z\}$ of the set $\Z$ induces a partition $\{\bar P\,,\bar Q\,,\bar Z\}$ of $\beta \Z$, given by their closures in $\beta \Z$.
Let $p\in \bar P$, we describe the map $\rho_p:X_f\to X_f$ giving its action on $X_f$. Let $U\in p$. Note that for arbitrarily large $k\in \N$,
$U\cap [a_k,b_k]\ne \emptyset$, where $[a_k,b_k]\subset P$. Thus given any $m\in \N$, select $k_m\in \Z$ and $t_{(U,k_m)}\in U\cap [a_k,b_k]$, where
$k_m$ is `slightly less' than one third $n_{k+1}-n_k$. Note that the net $\{t_{(U,k_m)}\}$ converges to $p$ and
$$
x^*p (t) = \lim\limits_{t_{(U,k_m)}} x^*t_{(U,k_m)}(t) = \lim\limits_{t_{(U,k_m)}} x^*(t_{(U,{k_m})}+t) = 1\,.\quad \text {for all}\ |t| <m\,.
$$
Since $n_{j+1}-n_j\to \infty$, it follows that $x^*p(t) = 1$ for all $t\in \Z$. Denoting the constant sequences $1$, $-1$ and $0$ by $\bf {1}$, 
$\bf {-1}$ and $\bf {0}$ respectively, we have shown that $x^*\cdot p \equiv f_p = \bf {1}$, if $p\in \bar P$. Actually the same argument is
valid for any translate of $f$ as well. Thus, $(x^*t)p = \bf {1}$ for all $t\in \Z$. Similarly, we can see that $(x^*t)r = \bf {-1}$ if $r\in \bar Q$ and
$(x^*t)r = \bf {0}$, if $r\in \bar Z$, for all $t\in \Z$. The last fact can be proved similarly, by considering a net, (or a sequence), $k_\ell \to -\infty$
and observing that $k_\ell \to r\in \bar Z$ and arguing as above. This shows that the only elements in the orbit closure of $x^*$, under the action of
$E(X_f,\Z)$ are $\bf {1}$, $\bf {-1}$ and $\bf {0}$, i.e. $X_f = \{x^*t\ |\ t\in \Z\}\cup \{{\bm 1}\,,{\bf {-1}}\,,{\bm 0}\}$. Now it is easy to verify
that each element $r\in \beta \Z$ fixes these three elements. Thus, we have a complete description of the elements of $E(X_f,\Z)\backslash \Z$,
they are the maps $\hat p$, $\hat q$ and $\hat z$ given by,
\begin{align}
(x^*t)\hat p & = \bf {1}\,,\quad \text {and}\ \hat p\ \text {fixes}\ \bf {1}\,,{-1}\,,\bf {0}\,, \notag \\
(x^*t)\hat q & = \bf {-1}\,,\quad \text {and}\ \hat q\ \text {fixes}\ \bf {1}\,,{-1}\,,\bf {0}\,, \notag \\
(x^*t)\hat z & = \bf {0}\,,\quad \text {and}\ \hat z\ \text {fixes}\ \bf {1}\,,{-1}\,,\bf {0}\,, \notag \\
\end{align}
where $t\in \Z$.

Finally, to show that $(X_f,\Z)$ is not weakly almost periodic, pick a sequence $k_\ell \to -\infty$ and $p\in \bar P$. Then $(x^*k_\ell)p = 1$,
for each $k_\ell$. But $\big(\lim\limits_{k_\ell \to -\infty} x^*k_\ell\big)\cdot p = {\bf{0}}\cdot p = \bf{0}$. This shows
that the map $\rho_p$, is discontinuous at $\bf {0}$. Thus $(X_f,\Z)$ cannot be weakly almost periodic, (since for such systems all
elements of the enveloping semigroup are continuous (see \cite {EN2})). 
\end{proof}

\smallskip

In fact, the following more general observation proves that countable, compact dynamical systems are tame.

\smallskip

\begin{lem}\label{Wu} Let $(X,T)$ be a compact countable dynamical system. Then $(X,T)$ is tame.
\end{lem}
\begin{proof} Let $f:X\to \R$ be any map. We want to show that $f$ is of Baire class one. This will show that any element of the enveloping
semigroup $E(X,T)$ is a Baire class one function and hence $(X,T)$ is tame. We need to show that the set $S \equiv \{x\in X\ |\ f\ \text {is
not continuous at}\ x\}$ is a set of first category. If $S$ is finite, this is obvious. So suppose $S$ is countable, say $S = \{y_j\ |\ j\in \N\}$.
Note that for any $j$, ${\overline {\{y_j\}}}^0 = \{y_j\}^0 = \emptyset$. If not, then $y_j$ is an isolated point and hence is not a point of
discontinuity of $f$. Thus, $S$ is of first category.
\end{proof}

\smallskip
\begin{rem}Of-course, the proof of Lemma \ref{Wu} can be obtained directly by applying Bourgain-Fremelin-Talagrand
dichotomy theorem, since the cardinality of $\beta T$ is at most $2^{\aleph_0}$. But, here, our arguments are much simpler. 
\end{rem}

The next result describes the nature of minimal sets and the support of an invariant ergodic measure on a Veech system. The proof presented by
W. Veech in \cite {V2} is primarily for the special case $(X_f,T)$, where $f\in \K(T)$. To prove analogous result for general Veech systems
we need to modify  arguments and use the enveloping semigroup machinery.

\smallskip

\begin{thm} \label{str-min-set} Let $(X,T)$ be a Veech system. Let $\mu$ be any invariant ergodic, Borel probability measure on $X$ with support $C(\mu )$. Then
\begin{enumerate}[label=(\arabic*)]
\item  $C(\mu)$ is a $\tau_1$-minimal set.
\item  In addition $T$ be amenable. Then every minimal subset of $X$ is almost automorphic, (in particular point distal) and is an `isometric extension'
(in the sense of \cite {DG}), i.e. is a measure theoretic isometric extension of its maximal equicontinuous factor.
\item  With $T$ amenable, every ergodic invariant measure on $X$ has discrete spectrum.
\end{enumerate}
\end{thm}
\begin{proof} (1): Let $\Big\{x_m\ |\ m\in \N \Big\}$ be a countable $\tau_2$-dense subset of $X$. Fix any $\eps >0$. Then $\Big\{B_\eps (x_m)\ |\ m\in \N\Big\}$ is a
cover of $X$, (recall that $B_\eps (x_m)$ is the $\eps$ ball centered at $x_m$ in metric $D$). Let $\Sigma (\eps )\subset \N$ be a
countable set such that $m\in \Sigma (\eps)$ if and only if $\mu (C_m) > 0$ where, $C_m(\eps) = B_\eps (x_m)\cap C(\mu )$. Then
$C(\eps ) = \bigcup\limits_{m\in \Sigma (\eps )}C_m(\eps )$ is a $\tau_1$-Borel subset of $C(\mu )$ of measure $1$ for every $\eps >0$.

Since $\mu$ is ergodic, by Lemma \ref{rec-lem}, there exists a $y_m\in C_m$ and a syndetic set $S_m\subset T$ such that $t\in S_m \equiv S_m(\eps )$ implies
$y_mt\in C_m$.\\

\noindent Claim : If $y\in C_m(\eps )$ and $t\in S_m$, then $D(y,yt)<3\eps$. This follows from the $T$ invariance of metric $D$ and following triangle inequality
$$
D(y,yt) \leq D(y,y_m) + D(y_m,y_m\cdot t) + D(y_m t ,yt) < 3\eps\,.
$$
Let $C_1 = \bigcap\limits_{n\in \N}C(\frac {1}{n})$. Then $\mu (C_1) = 1$ and if $y\in C_1(\mu )$, we have shown that the orbit of $y$ returns to
its $3\eps$ neighbourhood, (in $D$ metric), in a syndetic set, for every $\eps >0$. In particular, since $\tau_1\leq \tau_2$, it returns to its
given $\tau_1$-neighbourhood in a syndetic set. Since $(X,\tau_1)$ is compact, this means that $y$ is a $\tau_1$-almost periodic point, i.e. its
$\tau_1$-orbit closure is a $\tau_1$-minimal set. Note that we cannot say this with respect to the 
$\tau_2$ topology. To conclude the 
$\tau_2$-compactness of the $\tau_2$-orbit closure one would need some additional special properties, such as local compactness of the $\tau_2$-topology,
which in general we do not have.\\

\noindent (2): In fact, we can improve the previous claim to : If $y\in C_m(\eps )$ and $t,s\in S_m$, then $D(y,y(ts^{-1}))<4\eps$. This follows
from the inequality,
\begin{align}
D(y,y(ts^{-1})) & \leq D(y,y_m) + D(y_m,y_m(ts^{-1})) + D(y_m(ts^{-1}) ,y(ts^{-1})) \notag \\
& \leq D(y,y_m) + D(y_ms,y_mt) + D(y_m,y) \notag \\
& \leq D(y,y_m) + D(y_ms,y) + D(y,y_mt) + D(y_m,y) < 4\eps\,.\notag
\end{align}
This shows that every point in $C_1$ returns to its $4\eps$ neighbourhood in a set of times which is a $\Delta^*$ set, for every $\eps > 0$.
As before, since $\tau_1 \leq \tau_2$, the return time set of any point $y\in C_1$ to any of its $\tau_1$- neighbourhood is a $\Delta^*$ set. Hence
such a $y$ is $\tau_1$-almost automorphic.

Now , let
$$
C_2 = \Big\{y\in C(\mu )\ |\  \tau_1\text {-orbit closure of}\  y\ \text {equals}\ C(\mu ) \Big\}\,.
$$
Since $\mu$ is ergodic, $\mu (C_2) = 1$. Let $C = C_1\cap C_2\subset C(\mu )$. Then $\mu (C) = 1$ and if $y\in C$, then $\tau_1$-orbit closure
of $y$ is $C(\mu)$ and $y$ is almost automorphic. Even though $\mu (C(\mu )\backslash C) = 0$, unfortunately, for general Veech systems we are
unable to show that $C(\mu )\backslash C$ is the empty set. This would prove that $C(\mu )$ is actually minimal equicontinuous. We shall later
prove this for the special case of $(X_f,T)$, with $f\in \K(T)$.\\

We have shown that $C(\mu )$ is an almost automorphic minimal set. Now, we observe that it is a `regular' almost automorphic set and hence 
$(C(\mu ),T,\mu )$ is an measure theoretical `isometric extension' of its maximal equicontinuous factor, (see \cite {FGL} and \cite {DG}
for these notions). First, note that since $T$ is amenable, by a well known theorem of D. McMahon the regional proximality relation $Q(C(\mu ))$
on $C(\mu )$ is an `icer', i.e. invariant, closed equivalence relation. Next, we recall the `Veech relation' $V(Y,T)$ on any dynamical system $(Y,T)$, 
(see \cite {AGN}),
$$
V(Y,T) = \Big\{(y_1,y_2)\in Y\times Y\ |\ \text {there exists a net}\ t_\alpha\in T\ \text {and}\ z\in Y\ \text {such that}\ y_1t_\alpha\to z\ \text {and}\ zt^{-1}_\alpha \to y_2 \Big\}\,.
$$
Since each $x\in C\subset C(\mu)$ is almost automorphic, $V[x] \equaldef \{y\in X\ |\ (x,y)\in V(X,T)\} = \{x\}$. By Theorem 13
of \cite {AGN}, the cell $V[x]$ is dense in the cell $Q[x]$ of the regional proximality relation. Thus, if $\pi :C(\mu ) : \to C(\mu)/Q(C(\mu ))$
is the canonical factor map from $C(\mu )$ onto its maximal equicontinuous factor, then $\pi^{-1}(\pi (x)) = \{x\}$ for all $x\in C$. Thus $\pi$ is
one to one on set $C$, a set of full measure. Thus, $C(\mu )$ is an isometric extension of its maximal equicontinuous factor.\\

Finally, since $T$ is amenable, every minimal set $M$ is the support $C(\mu )$ of some ergodic invariant measure $\mu$. It follows that $(M,T)$
is minimal almost automorphic and is an isometric extension of its maximal equicontinuous factor.\\

\noindent (3): This immediately follows from (2), since $(C(\mu ),T,\mu )$ is measure theoretically isomorphic to a minimal equicontinuous system,
namely, its maximal equicontinuous factor.
\end{proof}

\smallskip

\begin{rem} The above theorem describes the structure of minimal sets in a general Veech system. However, for such systems (i) we cannot say much
about the regularity properties of the elements of its enveloping semigroup and (ii) in general the discrete nature of the spectrum cannot be
easily extended to non-ergodic measures. Now we shall show that for the special case of $(X_f,T)$, ($f\in \K(T)$), more can be said regarding
these two issues. 
\end{rem}

\smallskip

\begin{thm} Let $T$ be a countable group and $f\in \K(T)$. Then every element of $E(X_f,T)$ is Borel.
\end{thm}

\begin{proof} We start with a general Veech system $(X,T)$. Recall that $D$ is the metric on $X$ generating the $\tau_2$-topology. Let $p\in E(X,T)$
and let $\rho_p :X\to X$ be $\rho_p(x)= xp$. We show that $\rho_p$ is a $\tau_1$-Borel map. It is enough to show that if
$U\subset X$ is $\tau_1$-open, then $\rho_p^{-1}(U)$ is $\tau_1$-Borel.\\

For each $y\in U$ let $\eps \equiv \eps (y) >0$ be such that $B_\eps(y) \subset U$, (recall 
that $B_\eps(y) = \big\{y\in X\ |\ D(x,y)<\eps \big\}$)
and this choice is possible since $\tau_1 \leq \tau_2$). Since $(X,D)$ is separable, there exists a countable set
$\big\{y_n\ |\ n\in \N\big\}\subset U$ such that $U =\ds \bigcup_{y\in U}B_\eps (y) = \ds \bigcup_{n\in \N}B_\eps (y_n)$. Thus,
$$
\rho_p^{-1}(U) = \rho_p^{-1}\big( \bigcup_{n=1}^{\infty}B_\eps (y_n)\big) = \bigcup_{n=1}^{\infty}\rho_p^{-1}(B_\eps (y_n))\,.
$$
Therefore, it is enough to show that $\rho_p^{-1}(B_\eps (y_n))$ is $\tau_1$-Borel for each $y_n$. Next, let
$$
\Sigma_n = \rho_p^{-1}(y_n) = \big\{z\in X\ |\ zp = y_n \big\}\,.
$$
So far $(X,T)$ was a general Veech system. The following lemma is where we restrict to the case $X = X_f$, $f\in \K(T)$.

\smallskip

\begin{lem} \label{keylemma} With the notation as above, $\rho_p^{-1}(B_\eps (y_n)) = \bigcup_{z\in \Sigma_n}B_\eps (z)$.
\end{lem}

\smallskip

Assuming this lemma and again using separability of $(X,D)$, we can write $\bigcup_{z\in \Sigma_n}B_\eps (z)$ as a countable union of such balls.
Since each ball in the $D$ metric is a $\tau_1$-Borel set, $\rho_P^{-1}(B_\eps(y_n))$ is Borel for each $n\in \N$ and the proof is complete.
\end{proof}

\smallskip

\noindent {\bf Proof of Lemma \ref{keylemma}:}
\\
In this proof $\langle x,\xi\rangle$ will denote the `pairing' of vectors $x\in \ell^\infty(T)$ and $\xi \in \ell^1(T)$ as vectors in dual space.
First we claim that
\begin{equation}
\label{contraction}
\big\|xp-yp\big\|_\infty \leq \big\|x-y\big\|_\infty\ \text {for any}\ x,y\in X_f\subset \ell^{\infty}(T)\,.
\end{equation}
Let $t_\alpha \to p$ in $\beta T$, where $\{t_\alpha\}$ is a net in $T$. Consider,
\begin{align}
\big\|x-y\big\|_\infty &= \big\|xt_\alpha - yt_\alpha\big\|_\infty \,, \quad (\text {since}\ T\ \text{action preserves the}\ \ell^{\infty}\  \text {metric}) \notag \\
& \geq |\langle xt_\alpha - yt_\alpha ,z\rangle | \geq |\langle xp - yp,z\rangle |\,, \notag
\end{align}
where $z\in \ell^1(T)$ is any vector with $\big\|z\big\|_1 \leq 1$. Now we pick $z$ such that $|\langle xp - yp,z\rangle | = \big\|xp - yp\big\|_\infty$ and the claim
is proved. This claim implies $B_\eps(z) \subset \rho_p^{-1}(B_\eps (y_n))$ for each $z\in \Sigma_n$.

To prove the reverse inclusion, we need to show that $\bigcap_{z\in \Sigma_n}B_\eps (z)^c \subset \big(\rho_p^{-1}(B_\eps (y_n))\big)^c$,
where $A^c$ denotes the complement of set $A$. Let $x\in \bigcap_{z\in \Sigma_n}B_\eps (z)^c$. Thus $\big\|x-z\big\|_\infty \geq \eps$ for any
$z\in \Sigma_n$. Now, for any $z\in \Sigma_n$ and $\xi\in \ell^1(T)$ with $\big\|\xi\big\|_1 \leq 1$, we have
$$
|\langle (x-z)p,\xi\rangle | = \lim\limits_{\alpha}|\langle xt_\alpha -zt_\alpha ,\xi \rangle|\,.
$$
Whence, for each $\alpha$ we can choose $\xi_\alpha$ with $\big\|\xi_\alpha \big\|_\infty \leq 1$ such that
$\big\|xt_\alpha - zt_\alpha \big\|_\infty = |\langle (x-z)t_\alpha,\xi_\alpha \rangle |$.
Thus,
$$
|\langle (x-z)p,\xi_\alpha \rangle | = \big\|(x-z)t_\alpha \big\|_\infty = \big\|x-z\big\|_\infty \geq \eps\,, \quad (\text {by the hypothesis})\,.
$$
Now, select $\xi^*\in \ell^1(T)$ with $\big\|\xi^*\big\| \leq 1$ such that
$$
\big\|xp-zp\big\|_\infty = |\langle (x-z)p,\xi^*\rangle | = sup \big\{|\langle (x-z)p,\xi\rangle| \ |\ ||\xi||_1\leq 1\big\},.
$$
Thus $\big\|(x-z)p\big\|_\infty \geq \big\|x-z\big\|_\infty \geq \eps$. This shows that $x\notin \rho_p^{-1}(B_\eps (y_n))$ and the proof is complete.
\endproof

\smallskip

Using the previous theorem we can now prove the following.

\smallskip

\begin{prop} Let $T$ be amenable. Then, any minimal set of $(X_f,T)$ is equicontinuous, where $f\in \K(T)$.
\end{prop}

\begin{proof} We recall the notation used in the proof of Theorem \ref{str-min-set}. Any minimal set can be taken to be of the form $C(\mu )$ for
some ergodic invariant measure $\mu$. We need to show that $C(\mu ) = C$. Suppose this is not true. Then pick $x\in C$ and $y\notin C$.
Since $(C(\mu ),T)$ is minimal, there exists $p\in \beta T$ such that $\rho_p(x) \equiv xp = y$. Since the map $\rho_p:X_f\to X_f$ contracts
the $\ell^\infty$ metric on $X_f$, (which is the metric $D$ in the notation of Theorem (\ref{str-min-set})), for any $\eps > 0$,
$\rho_p(B_\eps (x)) \subset B_\eps (\rho_p(x)) = B_\eps(y)$, (recall that $B_\eps (x)$ denotes the $\eps$ ball in metric $D$ centered at $x$).
Note that since $y\notin C$, $\mu (B_\eps (y)) = 0$ for all small enough positive $\eps$'s and since $x\in C$, $\mu (B_\eps (x)) > 0$ for
all positive $\eps$'s. Now we show that since $\rho_p$ is Borel, it preserves $\mu$ and this will lead to a contradiction.  

Consider the map $\eta : \beta T \to \Omega_\mu (C(\mu )) : p\to U_p$, where $U_p[f] = [f]p$, ($f\in L^2(C(\mu ),\mu)$). Since $\rho_p :X\to X$
is Borel, $U_p[f] = [f]_p = [f\circ \rho_p] = [f_p]$. Given $\delta > 0$ and any Borel set $A\subset X_f$, consider the open
neighborhood $W_{A,\delta}$ of $p$ in $\beta T$ defined by
$$
W_{A,\delta} = \{q\in \beta T\ |\ \big|\langle U_q \chi_A,1\rangle - \langle U_{p}\chi_A\,,1\rangle \big| < \delta\}\,.
$$
Pick $t\in W_{A,\delta}$ and note that
\begin{align}
\langle U_t\chi_A\,,1\rangle & = \int_{C(\mu )}\chi_A(\omega t)d\mu (\omega ) = \mu (At^{-1}) = \mu (A)\,,\quad \text {and} \notag \\
\langle U_{p}\chi_A,1\rangle & = \int_{C(\mu )}\chi_A\circ \rho_{p}(\omega )d\mu = \mu (\rho_{p}^{-1}(A))\,. 
\end{align}
Thus, $\big|\mu (\rho^{-1}_{p}(A)) - \mu (A)\big| < \delta$. Since $\delta$ is arbitrary, $\mu (\rho^{-1}_{p}(A)) = \mu (A)$.
Now if $\eps > 0$ is small enough, using the fact that $B_\eps (x)\subset \rho^{-1}_{p}(B_\eps (y))$, we have
$$
0 = \mu (B_\eps(y)) = \mu (\rho^{-1}_{p}(B_\eps (y)) \geq \mu (B_\eps (x)) > 0\,,
$$
a contradiction. Thus $C(\mu ) = C$ and each point of $C(\mu )$ is almost automorphic, $C(\mu )$ being minimal, it follows that $(C(\mu ),T)$
is equicontinuous, (see \cite {AGN} Corollary 8).
\end{proof}

\smallskip

\begin{rem} Next we study the spectral feature of an invariant measure on $(X_f,T)$, ($f\in \K(T)$). Professor W. Veech had posed the question :
`Is the Sarnak conjecture valid for the flow $(X_f,T)$'?. This question is answered affirmatively if one shows that every invariant measure
has discrete spectrum. In a recent paper, (see \cite {HWY}) the authors attempt to give a proof of this for $T = \Z$. But to us, the proof
appears to be incomplete! We shall discuss the underlying issues with their proof and shall present a different proof. Thus proving
Sarnak conjecture for $\K(T)$, for any countable amenable $T$.
\end{rem}

\smallskip

\begin{thm} \label{dis-spect-Veech}  Any invariant measure on $(X_f,T)$, $f\in \K(T)$ has discrete spectrum.
\end{thm}

\smallskip

\noindent {\bf A discussion on the proof.}

\smallskip

Consider a general Veech system $(X,T)$ and let $\big\{x_m\ |\ m\in \N\big\}$ be a $\tau_2$-dense subset of $X$. Using $\tau_1$-compactness of $X$,
given any sequence $\{t_n\} \in T$, by the `diagonal argument' we can pick a subsequence $\{t_{n_k}\}$ such that the sequence $\{x_mt_{n_k}\}$
is $\tau_1$-convergent for each $m\in \N$. The key issue is to show that the sequence $\{xt_{n_k}\}$ is $\tau_1$-convergent for each $x\in X$.

To do this one needs to use the special structure given by the $T$-invariant metric $D$ generating the $\tau_2$-topology. Note that by viewing
$\{t_n\}$ as a net in $\beta T$ there is a subnet, (which may not be a subsequence), that converges to some $q\in \beta T$. Since
$\{x_mt_{n_k}\}_{k\in \N}$ converges for each $m$, it follows that it must $\tau_1$-converge to $x_mq$. Now we make a note of the following points

\noindent (1) We know that for each $x\in X$, there is a subnet of $\{xt_{n_k}\}$ that $\tau_1$-converges to $xq$ and this subnet will
depend on $x$. The crucial point is to show that the sequence $\{xt_{n_k}\}$ itself $\tau_1$- converges to $xq$ for each $x$.

\noindent (2) To do this, one may think of using the following triangle inequality,
$$
D(xt_{n_k},xq) \leq D(xt_{n_k},x_mt_{n_k}) + D(x_mt_{n_k},x_mq) + D(x_mq,xq)\,,
$$
and try to show that each terms on the right hand side gets small as $n_k\to \infty$. Convergence in $D$ metric will yield $\tau_1$-convergence.

\smallskip

\noindent ($2_a$) One has to be careful about `interchanging the limits'. That is, suppose $x_m\to x$ in the $\tau_1$ topology, in general
$\lim\limits_{m\to \infty}\lim\limits_{k\to \infty}x_mt_{n_k}$ may not exist and even if it does, may not be equal to
$\lim\limits_{k\to \infty}\lim\limits_{m\to \infty}x_mt_{n_k}$. Of course, the second limit exists and it is equal to $xq$. One could do this if $(X,T)$ is
weakly almost periodic, (\`a la `Grothendieck', see \cite {EN1}), but not for a general Veech system. Thus, for a general Veech system making
the third term $D(x_m q,xq)$ small is a problem. We have proved, (see Lemma \ref{keylemma}), that for $(X_f,T)$, ($f\in \K(T)$), the map $\rho_q$
is not only Borel but it is in fact $D$ contracting. This will enable us to make the third term small as $n_k\to \infty$.

\smallskip

\noindent ($2_b$) Making the second term small is even more problematic, because $x_mt_{n_k}\to x_mq$ only in $\tau_1$-topology. This is due to
$\tau_1$-compactness of $X$. The $\tau_2$-topology given by the metric $D$ is not compact. This is a real hurdle in directly proving that
$x{t_{n_k}}\to xq$. A way out is to work with continuous functions on $X$ rather than $X$ itself. We shall follow this approach, as in \cite {HWY}.

\noindent ($2_c$) The first term in the above triangle inequality is exactly where one uses the $T$ invariance of metric $D$. However, just making
these terms small in $D$ metric will not be enough, we need to do this in the $\tau_1$-topology, to use the $\tau_1$-compactness of $X$. We also need 
a `certain uniformity' to get rid of the dependence on sequence $\{t_{n_k}\}$.

Thus, summarizing, to get $xt_{n_k}$ $\tau_1$-converge to $xq$, we need (a) a certain `uniform mechanism' that will give us `$\tau_1$-closeness'
from `$\tau_2$-closeness'. This will be used after making the first and the third term small in $D$ metric. (b) To make the second term small,
we have to abandon the above triangle inequality and consider its analogue `for a continuous function'.

We again point out that the authors of \cite{HWY} tacitly move pass the above issues by claiming `it is not hard', (see \cite[p.849] {HWY}),
without giving any indication of how to resolve these issues. This makes their proof of Theorem (5.1) incomplete. We shall prove why $\{xt_{n_k}\}$
converges for each $x\in X$ for the system $(X_f,T)$, $f\in \K(T)$ and for general Veech systems provided they satisfy an additional
`uniformity condition'. Now we introduce this additional condition that the topologies $\tau_1$ and $\tau_2$ have to satisfy in order to
carry out the above line of argument and this will lead to showing that any invariant measure on such systems has discrete spectrum.

\smallskip

\begin{defn} A Veech system $(X,T)$ is said to be a strongly Veech if in addition to the five properties in the definition of
Veech systems, we also have the following sixth property:
\begin{itemize}
\item [(6)] Given a $\tau_1$-open set $V\subset X\times X$ containing the diagonal $\Delta_X$, there exists a $\delta > 0$
such that $B_\delta (x)\times B_\delta (x) \subset V$ for all $x\in X$.
\end{itemize}
\end{defn}

\smallskip

\begin{lem} The Veech system $(X_{f^*}, T)$, $f^*\in \K(T)$ is strongly Veech.
\end{lem}
\begin{proof} First observe that, given a $f\in X_{f^*}\subset l^\infty (T)$, a typical $\tau_1$-open neighbourhood of $f$ is given by
$V_{g,\eta}(f)$, where $g\in l^1(T)$ and $\eta > 0$ and
$$
V_{g,\eta}(f) = \big\{h\in X_{f^*}\subset l^\infty (T)\ |\ \big|\langle h-f,g\rangle \big| < \eta \big\}\,,
$$
where $\langle \,\,,\, \rangle$ is the canonical pairing between vectors in $l^\infty (T)$ and $l^1(T)$.

Let $V\subset X_{f^*}\times X_{f^*}$ be a $\tau_1$-open set containing the diagonal. Pick a $\tau_1$-open set $V_1$ such that
$\Delta_{X_{f^*}} \subset V_1 \subset V$ and
$$
V_1 = \bigcup\limits_{i=1}^{\ell}V_{g_i,\eta_i}(f_i)\times V_{g_i,\eta_i}(f_i)\,,
$$
and $\{V_{g_i,\frac {\eta_i}{2}}(f_i)\ |\ 1\leq i \leq \ell\}$ is a cover of $X_{f^*}$. Compactness of $X_{f^*}$ makes this possible.

Now we claim that, given $f\in X_{f^*}$,  $B_\delta (f) \subset V_{g_i,\eta_i}(f_i)$ for some $i\in \{1,\cdots,\ell\}$, where
$0<\delta < \min \big\{\frac {\eta_i}{2||g_i||_\infty}\ |\ 1\leq i \leq \ell\big\}$. To see this, first pick an $i$ such that
$f\in V_{g_i,\frac {\eta_i}{2}}(f_i)$, let $h\in B_\delta (f)$ and observe that
\begin{align}
\big|\langle h-f_i,g_i\rangle \big| & \leq \big|\langle (h-f) + (f-f_i),g_i\rangle \big| \notag \\
& \leq \big|\langle h-f,g_i\rangle \big| + \big|\langle f-f_i,g_i\rangle \big| \notag \\
& \leq ||h-f||_\infty\,||g_i||_1 + \frac {\eta_i}{2}\,,\quad (\text {since}\ f\in V_{g_i,\frac {\eta_i}{2}}(f_i)) \notag \\
& \leq \delta||g_i||_1 + \frac {\eta_i}{2} \leq \frac {\eta_i}{2} + \frac {\eta_i}{2} = \eta_i\,. \notag
\end{align}
Hence $h\in V_{g_i,\frac {\eta_i}{2}}(f_i)$. Thus, $\B_\delta (f)\times B_\delta (f)\subset V_1\subset V$.
\end{proof}

\medskip

\noindent {\bf Proof of Theorem \ref{dis-spect-Veech}:}\\ 
It is enough to show that each $g\in C(X_f)\subset L^2(X_f,\mu )$ is $\mu$-compact vector. To do this we show that
given any sequence $\{t_n\}$ in $T$, it has a subsequence $\{t_{n_k}\}$ such that $g_{t_{n_k}}$ converges pointwise on $X_f$,
(and hence by the dominated convergence theorem, in the $L^2$ norm on $(X_f,\mu )$). This will prove $\mu$-compactness of $g$.

So, as discussed before, by the `diagonal procedure' select a subsequence $\{t_{n_k}\}$ such that the sequence
$x_mt_{n_k}$ converges, (as $k\to \infty$), for each $m\in \N$. Now a subnet of $\{t_{n_k}\}$ converges to some $q\in \beta T$,
(in the topology on $\beta T$). Since $\{x_mt_{n_k}\}$ converges, it will converge to $x_mq$, ($m\in \N$).

Now we show that the sequence $g(xt_{n_k})$ converges for each $x\in X_f$. So fix any $x\in X_f$ and let $\eps > 0$ be given.
Since $(x,y)\to g(x)-g(y)$ is continuous and $X$ is $\tau_1$-compact, we can find a $\tau_1$-open neighbourhood $V$ of the
diagonal $X_f\times X_f$ such that if $(x,y)\in V$ then $\big|g(x) - g(y)\big| < \frac {\eps}{3}$. For this $V$, pick $\delta > 0$
as in `Property (6)', (see the definition of strong Veech systems). Pick $m\in \N$ such that $D(x,x_m) < \delta$. Note that
(i) $D(x_mq,xq) \leq D(x_m,x)$ by (\ref{contraction}) and (ii) $D(xt_{n_k},x_mt_{n_k}) = D(x,x_m) < \delta$. Thus,
$\big(xt_{n_k},x_mt_{n_k}\big)\in V$ and $(x_mq,xq)\in V$. Now consider the inequality,
\begin{align}
\big| g(xt_{n_k}) - g(xq) \big| & \leq \big| g(xt_{n_k}) - g(x_mt_{n_k})\big| + \big| g(x_mt_{n_k}) -  g(x_mq) \big| + 
\big|g(x_mq) - g(xq)\big| \notag \\
& \leq \frac {\eps}{3} + \big| g(x_mt_{n_k}) - g(x_mq) \big| + \frac {\eps}{3}\,. \notag
\end{align}
Thus, there exists $k_0$ such that if $k > k_0$, then $\big| g(xt_{n_k}) - g(xq) \big| < \eps$. This proves pointwise
convergence of $g(xt_{n_k})$. 
\endproof

\smallskip

\begin{rem} Actually a tiny modification of the arguments in above proof yields the same conclusion for any strongly Veech system.
\end{rem}

\smallskip

\begin{thm} Let $T$ be amenable, then any invariant measure on a strongly Veech sytstem $(X,T)$ has discrete spectrum.
\end{thm}

\begin{proof} With the notation as in the previous theorem, we need to show that the sequence $g(xt_{n_k})$ converges for each $x\in X_f$.
We can show that it is a Cauchy sequence by considering the inequality
\begin{align}
\big| g(xt_{n_k}) - g(xt_{n_l}) \big| & \leq \big| g(xt_{n_k}) - g(x_mt_{n_k})\big| + \big| g(x_mt_{n_k}) -  g(x_mt_{n_l}) \big| + \big|g(x_mt_{n_l}) - g(xt_{n_l})\big| \,. \notag \\
& \leq \frac {\eps}{3} + \big| g(x_mt_{n_k}) - g(x_mt_{n_l}) \big| + \frac {\eps}{3}\,. \notag
\end{align}
The rest of the argument is as before.
\end{proof}

\smallskip

Finally, one would like to know whether $(X_f,T)$, ($f\in \K(T)$) is tame, or more generally any strongly Veech system is tame? We answer this question below.

\smallskip

\begin{thm} \label{Sarnak-for-Veech} Let $(X,T)$ be a strongly Veech system with $T$ amenable.
\begin{enumerate}[label=(\arabic*)]
\item  If $(X,\tau_1)$ is metrizable, then $(X,T)$ is tame.
\item  In particular $(X_f,T)$ is tame, where $f\in \K(T)$, (recall that $T$ is countable, amenable). \label{Misha}
\item  As a consequence, metrizable, strongly Veech systems have zero topological entropy and
\item  the Sarnak conjecture holds for such systems.\footnote{See Section 6. for more details.}
\end{enumerate}
\end{thm}
\begin{proof} (1): Our assumption implies that $C(X)$ the space of continuous real valued functions on $X$ with the sup-topology
is separable. Fix a countable dense set $g_n\in C(X)$. In the above theorem we have already shown that given any $\alpha\in \beta T\backslash T$,
and $g\in C(X)$, there exists a sequence $\{t_k\}$ in $T$ such that the sequence $g(xt_k)$ converges to $g(x\alpha )$ for all $x\in X$. Again,
by the arguments in the previous theorem, given any $\alpha \in \beta T\backslash T$, we can find a sequence $\{t_k\}$ such that $g_n(xt_k)$ converges to
$g_n(x\alpha )$ as $k\to \infty$, for each $x\in X$ and $n\in \N$. Since $\{g_n\ |\ n\in \N\}$ is dense in $C(X)$, this implies $xt_k\to x\alpha$,
(in $\tau_1$ topology), for each $x\in X$. This shows that $\rho_\alpha\in E(X,T)$ is of Baire class 1, for every $\alpha \in \beta T\backslash T$.
Thus $(X,T)$ is tame.\\

\noindent (2): We only need to observe that $(X_f,\tau_1)$ is metrizable. Note that since $T$ is countable and $f$ is bounded, with out loss of generality
$|f(t)| \leq 1$, $t\in T$. Let $\psi:X_f\to [0,1]^T$ be the map $(\psi (x))_t = xt$. Then $\psi$ is an injective map onto its image and it intertwines
the $T$ action on $X_f$ with the shift action on the Hilbert cube. Observe that $\psi$ is a homeomorphism where its domain has the $weak^*$ topology
and the range has the (restriction of) the product topology on $[0,1]^T$. The later topology being metric, it follows that $(X_,\tau_1)$ is metrizable.\\  

\noindent (3) and (4): Now these results follow from the fact that every invariant measure on $X$ has discrete spectrum and $X$ is metrizable.
\end{proof}

\smallskip

\begin{rem} After the first version of this paper was posted on Arxiv, M. Megrelishvili informed us that one can prove \ref{Misha} of
Theorem \ref{Sarnak-for-Veech}  by applying \cite[Theorem 8.2.4]{GM3} combined with \cite[Theorem 9.12]{GM3} and \cite[Theorem 6.1]{GM4}.However,
our proof for amenable acting groups is direct and self contained.
\end{rem}

\medskip

\section{Applications to number theory}

\medskip

\noindent {\bf M\"{o}bius disjointness.} In this section, we are interested in the applications of our results on Veech systems to
Number theory. Precisely, our applications are related to the so called M\"{o}bius randomness law as formulated by P. Sarnak in his
striking paper \cite{Sarnak}. This law is about the dynamical behavior of the M\"{o}bius and Liouville functions.

We recall that the integer is square-free if its prime decomposition does not contain any square. The Liouville function $\bml$ is
defined as $1$ if the number of the prime factor of the integer is even and $-1$ if not, and the M\"{o}bius function ${\bmu}$ coincide
with the Liouville function on its support which is the subset of square-free  integers.\\ 

The M\"{o}bius randomness law \`a la Sarnak state that the statistical average or C\'esaro average of the values of a continuous map
along a orbit of any point $x$ with respect to any transformation with zero topological entropy, averaged with weights given by the
M\"{o}bius function, converge to zero. Formally, this law can be stated as follows:

\noindent {\bf Sarnak's M\"obius disjointness Conjecture.} Let $(X,T)$ be a compact metric, topological dynamical system with topological
entropy zero, then, for any $x\in X$, for any continuous function $f:X\to \R$, the following should hold.
$$
\lim\limits_{N\to \infty}\frac{1}{N}\sum_{n=1}^{N}\bmu (n) f(T^nx) = 0\,.
$$

This law is also known as Sarnak's conjecture or  M\"{o}bius disjointness conjecture. We proved that the Sarnak conjecture holds for
the system $(X_f,\Z)$, where $f\in K(\Z)$. A bit later we shall see a number theoretic consequence of this. But first we recall that
for the simplest zero entropy dynamical system--the irrational rotation of the circle, Sarnak's conjecture is a consequence of
the following Davenport estimate,
(see  \cite{Da}),
\[
\max_{\theta \in \T}\left|\displaystyle\sum_{k \leq x}\bmu(k)e^{ik\theta}\right|
\leq \frac{x}{\log(x)^{A}},
\qquad {\rm where}\  A >0.
\]
We view this as `M\"{o}bius disjointness' for the almost periodic map $k\to e^{ik\theta} : \Z\to \R$. Now we can extend this
`M\"{o}bius disjointness' to Besicovitch almost periodic functions on $\Z$ by the following simple argument. Let $f:\Z\to \R$ be a
Besicovitch almost periodic map. Thus, there is a sequence $\{g_j\}$ of (Bohr) almost periodic maps from $\Z$ to $\R$ such that given
$\eps > 0$ there exists a $k\in \N$ such that $\big\|f-g_k\big\|_{B_1} < \eps$. Now for any $N\geq 1$, we have
\begin{align*}
\Big|\lim\limits_{N\to \infty}\frac{1}{N}\sum_{n=1}^{N}\bmu(n) f(n) - \lim\limits_{N\to \infty}\frac{1}{N}\sum_{n=1}^{N}\bmu(n) g_k(n)\Big|
&\leq \sup_{j}\Big(\frac{1}{N}\sum_{n =1}^{N}\Big|f(n)-g_k(n)\Big|\Big) \\
&= \big\|f-g_k\big\|_{B_1} < \eps \,.
\end{align*}

This extension of `M\"{o}bius disjointness' from Bohr almost periodic to Besicovitch almost periodic functions immediately yields the
following.

\smallskip

\begin{thm} Let $(X,T,\mu)$ be uniquely ergodic system with discrete spectrum. Then, the M\"{o}bius disjointness holds.
\end{thm}
\begin{proof} This follows immediately from Corollary \ref{Besicovitch}.
\end{proof}

\smallskip

As an immediate consequence, we have the following.

\smallskip

\begin{cor} The M\"{o}bius disjointness holds for any weakly almost periodic system. 
\end{cor}
\begin{proof} Weakly almost periodic systems are uniquely ergodic with discrete spectrum, (see \cite{EN2}).
\end{proof}

\medskip

\begin{rem} With our notion of weak tameness, Theorem 1.2 of \cite{HWY} says that the Sarnak conjecture holds for weakly tame systems. Here,  we have used this theorem to prove the validity of Sarnak conjecture for
Veech systems. Unfortunately, this theorem does not say anything about the validity of M\"{o}bius disjointness for the simpler example
\ref{Namioka}. Furthermore, even if that theorem is improved to establish M\"{o}bius disjointness for systems with only countably many
ergodic measures with discrete spectrum, it still does not say anything about example \ref{Namioka}. In addition, one also observes
that the results of a recent paper \cite{FN} do not apply to our example to validate `logarithmic M\"{o}bius disjointness'. On the other
hand it is easy to check that this example satisfy M\"{o}bius disjointness conjecture. Notice further that the results of a recent
paper \cite{FN} do not apply to the graph maps and dendrites maps.

In the forthcoming paper \cite{AM2}, the authors proved that Sarnak's M\"{o}bius disjointness holds if each invariant measure has a singular
spectrum. Therefore, it suffices to establish that the conjecture holds only for the system for which invariant measures that have a
Lebesgue component. We further establish that the spectral measure of the M\"{o}bius function is absolutely continuous with respect
to the Lebesgue measure. We would like also to point out that therein the authors present a `dissection of M\"{o}bius flow \`a la Veech'
and use the result of Rokhlin-Sinai which assert that for any dynamical system with positive entropy has the relatively Kolmogorov
property with respect to Pinsker algebra. This was also observed and popularized by Jean-Paul Thouvenot \cite{Th}.  Accordingly, now it is obvious
to deduce that the dynamical system $(x,y) \in \T \mapsto (x,x+y)$ can not be a factor of the M\"{o}bius flow.
\end{rem}

\noindent {\bf Improving Motohashi-Ramachandra estimate.}

\medskip

Here, we will gives a simple argument which yields a slight improvement of an old result of Motohashi-Ramachandra \cite{Mo}, \cite{Ram}
on the behavior of Mertens function $M(x)\equaldef\sum_{n \leq x}\bmu(n)$ on the short interval. We start by recalling Motohashi-Ramachandra's result.

\smallskip

\begin{lem}[Motohashi-Ramachandra's theorem \cite{Mo}, \cite{Ram}]\label{MoRam} The Mertens function satisfy, 
$$
\big|M(x+h)-M(x)\big|=o(h)\,,
$$
uniformly in $h$ satisfying $x^{\tau} \leq  h \leq x$, whenever $\tau > \frac{7}{12}$. 
\end{lem}

\smallskip

However, let us mention that in the same year, using the so-called Hooley-Huxley contour, K. Ramachandra obtain the following estimations.

\begin{lem} [Ramachandra's theorem \cite{Ram}] The Mertens function satisfy, for any $A>0$, 	
\begin{align}
\sum\limits_{x \leq n \leq x+h}\bmu(n) & = O_{\eps , A}\Big(\frac{h}{\log(x)^A} + x^{\frac7{12}+\eps }\Big)\quad \text {and}\,, \textrm{as~~}x \longrightarrow +\infty \label{R1} \\
\frac1{X}\int_{X}^{2X}\Big|\sum\limits_{x \leq n \leq x+h}\bmu (n)\Big|^2 dx & = O_{\eps ,A}\Big(\frac{h}{\log(X)^A} +  X^{\frac1{6}+\eps }\Big)\,, \textrm{as~~}X \longrightarrow +\infty \,.\label{R2}
\end{align}
\end{lem}
In his 2016's paper \cite{V2}, Professor W. Veech observes that no progress was made on the behavior of Mertens function in the short interval since Motohashi and
Ramachandra original papers. It turns out that in the same year, Matomaki-Radzwi\l{}l in \cite{MR} improved \eqref{R2} by establishing that for any  $\varepsilon>0$
and $h \leq X^{\varepsilon}$, we have

\begin{align}\label{MR}
	\frac1{X}\int_{X}^{2X}\Big|\sum\limits_{x \leq n \leq x+h}\bmu (n)\Big|^2 dx & =o(Xh^2).
\end{align}
 
\noindent{}Notice that it is easy to obtain the following corollary from Motohashi-Ramachandra's theorem.

\smallskip

\begin{cor} Let $(x_n)$ a sequence of positive real numbers and $\tau> \frac{7}{12}.$  Suppose that $x_n+(x_{n+1}-x_n)^\tau \leq x_{n+1}\leq 2x_n$,
for a large $n$. Then,
$$
\sum_{k=1}^{n}\big|M(x_{k+1})-M(x_k)\big|=o(x_{n+1})\,.
$$
\end{cor}

\smallskip
\noindent{}One can state similar corollary by applying \eqref{MR}.\\

\noindent{} Here, our Theorem (\ref{Sarnak-for-Veech}) showing that Sarnak conjecture holds for $(X_f,\Z )$, for $f\in \K(\Z)$, will allow us to obtain
a stronger result, namely the following.

\smallskip  

\begin{thm}\label{Motohoshi-ext} Let $(x_n)$ a sequence of positive real numbers such that $x_{n+1}-x_n\to +\infty$ as $n\to +\infty$. Then,
$$
\sum_{k=1}^{n}\big|M(x_{k+1})-M(x_k)\big|=o(x_{n+1})\,.
$$
\end{thm}
\begin{proof} Let $k\in \N$ and put 
\[
\epsilon_k=\begin{cases}

\sgn\big(M(x_{k+1})-M(x_k)\big) &\text{if }  M(x_{k+1})-M(x_k) \neq 0,\\

1, &\text{if ~not,} 

\end{cases}
\]
where $\sgn(x)=\frac{x}{|x|},$ for $x \neq 0$. Now, define a sequence $f = f_{(\eps_k)}$ by
$$
f(n)=\sum_{k \geq 1}\epsilon_k\1_{[x_k,x_{k+1})}(n)\,.
$$
Clearly $f$ is in $\ell^{\infty}(\Z )$ and as shown before, $f\in K(\Z)$. Since Sarnak's conjecture
holds for $(X_f,\Z)$,
\begin{equation}\label{eq1}
\sum\limits_{k=x_1}^{x_{n+1}}\bmu (k)f(k) = o(x_{n+1})\,.
\end{equation}
But
\begin{align}\label{eq2}
\sum\limits_{k=x_1}^{x_{n+1}}\bmu(k)f(k) & = \sum_{j=1}^{n}\sum_{k=x_j}^{x_{j+1}}\bmu(k)f(k) \notag \\
& = \sum_{j=1}^{n}\epsilon_j\sum_{k=x_j}^{x_{j+1}}\bmu(k)\notag \\
& = \sum_{j=1}^{n}\sgn\big(M(x_{j+1})-M(x_j)\big)\sum_{k=x_j}^{x_{j+1}}\bmu(k) \notag \\
& = \sum_{j=1}^{n}\Big| M(x_{j+1})-M(x_j) \Big|\,.	
\end{align}
The last inequalities follows from the definition of $(\epsilon_k)$ and $M$. Combining \eqref{eq1} and \eqref{eq2},
we obtain the desired estimation, and the proof of the proposition is complete.\footnote{This result can be obtained also as
a consequence of Matomaki-Radzwi\l\l's result \cite{MR}. However, our proof avoid the heavy analytic Number Theory machinery. }
\end{proof}

\smallskip
 \begin{rem} Note that once we show that translation flow on Veech function is tame, validity of Sarnak conjecture for this flow follows from to Theorem 2.1 from \cite{HWY}. It is a common misunderstanding that one uses the work of Motomaki-Radzwi\l\l \cite{MR} for this purpose. It is a result of Matomaki-Radzwi{\l}{\l}-Tao on the validity of averaged form of Chowla of order two, (see \cite{MRT}) was used in \cite{HWY} for the proof of Sarnak conjecture for systems for which every invariant measure has discrete spectrum. Let us notice further that this improvement can be  obtained also as a consequence of Matomaki-Radzwi{\l}{\l}'s result \cite{MR}. However, our proof avoid the heavy analytic Number Theory machinery.  We point out that the Motomaki-Radzwi{\l}{\l}-Tao result on the validity of averaged form of Chowla of order two, (in \cite{MRT} ), does not need a more elaborate machinery of analytic number theory like the result of Motomaki-Radzwi\l\l in \cite{MR}  as it is shown in the appendix. Indeed, the only ingredient needed for the proof is Davenport estimate. Thus, even though Theorem \ref{Motohoshi-ext} can be derived also from Matomaki-Radzwi{\l}{\l}  result of \cite{MR}, our approach considerably reduces the input from Number Theory.
	
Having said this, we also point out that our proof of Sarnak conjecture for systems with singular spectrum in reference \cite{AM2} bypasses even Motom\"{a}ki-Radzwi{\l}{\l}-Tao and makes our approach to Mertens's growth far more dynamical/ergodic-theoretic with only a minimal number theory input, (namely, uses only Davenport estimate).
\end{rem}
\smallskip
\noindent{}The previous result can be improved by assuming Chowla conjecture which asserts that for any distinct integers $s_1,s_2,\cdots,s_k$,
$k \geq 1$, we have
$$
\frac{1}{N}\sum_{n=1}^{N}\bml(n+s_1)\cdots\bml(n+s_k) \tend{N}{+\infty}0.
$$
It follows that for any distinct integers $a_1,a_2,\cdots,a_k$, $k \geq 1$, we have
$$
\frac{1}{N}\sum_{n=1}^{N}\bml^{a_1}(n+s_1)\cdots\bml^{a_k}(n+s_k) \tend{N}{+\infty}0\,.
$$

\noindent{}We thus get that $\bml$ is normal, that is, generic for the Bernoulli measure $dB(1/2)=\ds \otimes_{k \in \N}
(\frac12 \delta_{1}+\frac12 \delta_{-1})$. \\

\noindent{}Let $X=\{-1,1\}^{\N}$ and $X_{\bml}$ be the orbit generated by $\bml$ under the shift map $S~~:~~\omega \mapsto  
S(\omega)=(\omega{(n+1)})$. For any $\omega \in X_{\bml}$, we define the random Mertens function  by
$$
M_{\omega}(x)=\sum_{n \leq x}\omega(n)\,.
$$

\smallskip

\begin{defn} Let $f$ be an arithmetic function $(f:\N \longrightarrow \Cc)$ and $\tau_0 \geq 0$. $f$ is said to satisfy Motohashi-Ramachandra
property of order $\tau_0$ if for any $\tau>\tau_0$, we have
$$
\sum_{x}^{x+h}f(n)=o(h)\,,
$$
uniformly in $h \geq x^{\tau}$, that is, there is a sequence $(\delta_x)$ such that $(\delta_x) \tend{x}{+\infty}0,$ and there exist a constant
$C_\tau>0$ for which we have
$$
\sup_{x^{\tau} \leq h \leq x}\Big| \frac{1}{h} \sum_{x}^{x+h}f(n) \Big| \leq  C_\tau \delta_x\,.
$$
\end{defn} 

\smallskip

At this point, let us point out that Motohashi \cite{Mo} and independently Ramanchandra \cite{Ram} proved that the M\"{o}bius function
satisfy their property for $\tau_0=\frac{7}{12}$(Lemma \ref{MoRam}). Later, T. Zhan extended Motohashi-Ramachandra's result by proving
that the estimation of Davenport holds on the short interval \cite{Zhan}, that is,
$$
\sup_{x^{\tau} \leq h \leq x}\Big| \frac{1}{h} \sum_{x}^{x+h}\bmu(n) e^{i n \theta} \Big| \leq  \frac{C_{\tau,\epsilon}}{\log(x)^\epsilon}\,,
$$
for any  $\tau>\frac{5}{8}$ and $\epsilon>0$, uniformly in $\theta$. For the quadratic case, J. Liu and T. Zhan improved Hua's result on the
sum of two primes and a prime square in \cite{Hua} by establishing that the bound can be $\frac{11}{16}$ and $\frac{2}{3}$ under GRH \cite{LZ}.
Besides, D. Hajela and J. Smith conjectured, (see \cite{HS}), that for any $\tau>\frac12$,
$$
\sup_{\theta}\Big| \frac{1}{x^{\tau}} \sum_{ n \leq x}\bmu(n) e^{i n \theta} \Big| \leq  \delta_{\tau,x}\,,
$$
with $\delta_{\tau,x} \tend{x}{+\infty}0,$.\\

When this paper was in final preparation, Igor Shparlinski informed us that Matom\"{a}ki- Ter\"{a}v\"{a}inen. improved the
bounded to $11/20$ \cite{MT}. Here, in the spirit of the dichotomy \`a la Sarnak \cite{Video-Sarnak}, \cite{Asia-Sarnak} we establish the following:

\smallskip

\begin{thm}Under Chowla conjecture, we have for almost all $\omega \in X_{\bml}$, for any $\tau>\frac12$,
$$
\sup_{{x}^\tau \leq h \leq x} \Big|\frac{1}{h}\Big(M_\omega(x+h)-M_\omega(x)\Big)\Big| \tend{x}{+\infty}0\,.
$$
\end{thm}
\begin{proof}We start by noticing that under Chowla conjecture, the system $(X_{\bml},\mathcal{B},S,dB(\frac{1}{2}))$ is a Bernouilli system. Therefore,
the sequence of random Merstens function $(M_\omega(x))$ is a martingale. We thus get, by Doob-Klomogorov inequality 
\begin{eqnarray}
\Big\| \sup_{{x}^\tau \leq h \leq x}\Big|\frac{1}{h}\Big|M_{\omega}(x+h)-M_\omega(x)\Big| \Big\|_2
&\leq& \Big\| \sup_{{x}^\tau \leq h \leq x}\Big|\frac{1}{h}M_\omega(x+h)\Big| \Big\|_2+
\frac{1}{x^{\tau}}\Big\|M_\omega(x)\Big\|_2 \nonumber\\
&\leq& \frac{1}{x^{\tau}}\Big\|M_\omega(2x)\Big\|_2+\frac{1}{x^{\tau}}\Big\|M_\omega(x)\Big\|_2\\
&\leq& \big(\sqrt{2}+1\big)x^{\frac12-\tau}.
\end{eqnarray}	
Now, we  apply Etemadi's trick. Take $\tau>\frac12$ and $x=[\rho^y]$, $\rho>1$, to see that
$$
\sum_{x}  \Big\| \sup_{{x}^\tau \leq h \leq x}\Big|\frac{1}{h}\Big(M_x+h(\omega)-M_x(\omega)\Big)\Big| \Big\|_2<+\infty\,.
$$
This gives that, for almost all $\omega \in X_{\bml}$,
$$
\sup_{{x}^\tau \leq h \leq x}\Big|\frac{1}{h}\Big(M_\omega(x+h)-M_\omega(x)\Big)\Big|\tend{x}{+\infty}0\,.
$$
We finish the proof by letting $\rho \rightarrow 1$.
\end{proof}

\smallskip

\noindent {\bf Conjecture.} We conjecture that for any $\tau>\frac12$,
$$
\big|M(x+h)-M(x)\big|=o(h)\,,
$$
uniformly in $h$ provided $x^{\tau} \leq h \leq x$.


\medskip

\noindent {\bf Besicovitch almost periodicity of certain number theoretic functions.}

\medskip

Now, we would like to mention that G. Rauzy pointed out that the square of the M\"{o}bius function is a Besicovitch almost periodic
sequence (i.e. a Besicovitch almost periodic function), (see \cite[p.99]{Rauzy}). Here, let us notice that this fact can be extended to
the analogous number theoretic map in the more general setting of $\B$-free integers. We recall this notion of P. Erd\"{o}s \cite{Er}.

\smallskip

\begin{defn} Let $\B = \{b_k\ |\ k\in \N\} \subset\{n\in \N\ |\ n\geq 2\}$ be a subset of natural numbers which have
the following properties:
\begin{equation} \label{eq:rel-prime}
\text {for all}\ 1\le k < k',\ b_k\ \text {and}\  b_{k'}\ \text {are relatively prime and}\ \sum_{k\ge 1}\frac{1}{b_k} < \infty \,.
\end{equation}
Integers with no factors in $\B$ are called $\B$-free integers and the set of $\B$-free integers will be denoted by the set $\mathbb{B}$.
\end{defn}

\smallskip

Let $\chi_{\mathbb{B}}$ denote the indicator function of the set $\mathbb{B}$. The set of square-free integers is a special case when
$\B$ is the set of all squares primes. L. Mirsky had studied, (see \cite{Mi1}, \cite {Mi2}, \cite {Mi3}), the distribution of patterns
in the characteristic function of \emph{$r$-free numbers}, that is, the numbers which are not divisible by the $r$-th power
of any prime ($r\ge 2$).

To establish that the indicator function of $\B$-free numbers is a Besicovitch sequence, it suffices to prove that the
indicator function $\chi_{m_{\B}}$ of the subset $m_{\B} \setdef \Big\{x | x \equiv 0\ \text{mod}\ b_k\ \text{for some}\ k \geq 1  \Big\}$  is a Besicovitch
sequence. For that let $K \geq 1$ and $\chi_{m_{\B_K}}$ the indicator function of the subset 
$m_{\B_K} \setdef \Big\{x | x \equiv 0\ \text{mod}\ b_k \ \text{for some}\ k \in \big\{1,\cdots, K\big\}\Big\}$. It follows that 

\begin{equation} \label{approxi}
\limsup \frac{1}{N}\sum_{n=1}^{N}|\chi_{m_{\B}}(n)-\chi_{m_{\B_K}}(n)| \leq \sum\limits_{k >K}\frac{1}{b_k} \to 0\,,\ \text {as}\ K\to \infty.
\end{equation}

Furthermore, $\chi_{m_{\B_K}}$ is a periodic function. Taking into account that Mirsky's theorem can be extended to $\B$-free integers ([4]),
(that is, the indicator function of $\B$-free integers is a `generic point' for the Mirsky measure), our Theorem 3.11 shows that the subshift
generated by $\chi_{\mathbb{B}}$ its Mirsky measure has discrete spectrum. This gives a new and simple proof of Cellarosi-Sinai theorem \cite{CS}
and el Abdalaoui-Lema\'{n}czyk-de-la-Rue extension of it \cite{ALR}.\\

We need to point out here that the principal tool in the proof of Mirsky theorem is based on the notion of admissibility.
This notion is  crucial in the studies of the dynamical behavior of $\B$-free systems. It is also fundamental in the
structure of M\"{o}bius flow and the well-know Chowla conjecture. For more details, we refer to \cite{AV}. \\

We recall that the subset $A$ of positive integers is \emph{$\B$-admissible} if for any $k \geq 1$, the image of
$A$ under the maps $x \in \N^* \mapsto \overline{x}\in\Z/b_k\Z$ is proper, that is,
$$
\left|\bigl\{y\in\Z/b_k\Z: \exists n\in A, n=y\ [b_k]\bigr\}\right|<b_k\,.
$$
An infinite sequence $x=(x_n)_{n\in\N^*}\in \{0,1\}^\N$ is said to be \emph{$\B$-admissible} if its \emph{support}
$\{n\in\N^*: x_n=1\}$ is $\B$-admissible. In the same way, a finite block $x_1\ldots x_N\in\{0,1\}^N$ is \emph{$\B$-admissible} if
$\{n\in\{1,\ldots,N\}: x_n=1\}$ is $\B$-admissible.

Let us notice that the approximation of $\chi_{\mathbb{B}}$ by the periodic function $\chi_{m_{\B_K}}$ can not be uniform in the following sense
$$
\limsup_{N}\sup_{k} \Big(\frac{1}{N}\sum_{n=1}^{N}|\chi_{\mathbb{B}}(n+k)-\chi_{m_{\B_K}}(n+k)|\Big)=0\,,
$$
since the flow generated by the indicator function of $\B$-free numbers has a positive topological entropy. We can also
see this directly. Indeed, for any $x >0$, the sequence $\underbrace{00\cdots0}_{[x] \text{~times}}$ is an admissible sequence.
Therefore, for any fixed $x$ there  is a positive density of $k$'s for which $\chi_{\mathbb{B}}(n+k)=0,$  for $n=1,\cdots, [x]$.
Moreover, if $k$ is a multiple of the period $c$ of $\chi_{m_{\B_K}}$, then we have $\chi_{m_{\B_K}}(n+k)=\chi_{m_{\B_K}}(n)$. Thus, we get
\begin{align}
\limsup_{N}\sup_{k} \Big(\frac{1}{N}\sum_{n=1}^{N}| & \chi_{\mathbb{B}}(n+k)-\chi_{m_{\B_K}}(n+k)|\Big) \notag \\
& \geq \limsup_{N}\Big(\frac{1}{N}\sum_{n=1}^{N}|\chi_{m_{\B}}(n+k.c)-\chi_{m_{\B_K}}(n+k.c)|\Big) \notag \\
& \geq \limsup_{N}\Big(\frac{1}{N}\sum_{n=1}^{N}\chi_{m_{\B_K}}(n)\Big) \notag \\
& = \prod_{k=1}^{K}\Big(1-\frac{1}{b_k}\Big)>0\,. \notag 
\end{align}
\textbf{Acknowledgment.}

The authors wish to express their thanks to Ahmed Bouziad for a discussion on the subject. We also thank Eli Glasner and M. Megrelishvili
for their comments and criticism which improved the earlier version of the paper. The first author would like to thanks Igor Shparlinski
for a discussion on the subject. He gratefully thanks the Rutgers University and the Institute of Mathematical Sciences at Chennai for
their hospitality. Both authors would like to express their gratitude to the Vivekananda Institute in Calcutta for their invitation and
hospitality where part of this work was done.

\appendix
\section*{Appendix. On the average Chowla of order two.}
In this short note, by applying a Bourgain's observation, we present a simple proof of Matom\"{a}ki-Radziwi{\l}{\l}-Tao theorem on the average Chowla of order two \cite{MRT} based on Davenport theorem. In their inequality $H$ is allowed to grow very slowly with respect to $X$. Here, for $H=X$ we obtain a bound for the speed of convergence. Notice that this is the only ingredient needed for the proof of the validity of Sarnak conjecture for systems with discrete spectrum in Huang-Wang-Ye's result.
\begin{thm}\label{Main-Speed1}Let $\bnu$ be a M\"{o}bius or Liouville function.
	Then, for any $N \geq 2$,
	\begin{align*}
	&\frac1{N}\sum_{m=1}^{N}\Big|\frac1{N}\sum_{n=1}^{N} \bnu(n) \bnu(n+m)\Big| 
	\\&\leq \frac{C}{\log(N)^{\kappa}}.
	\end{align*}	
	where $C$ is some positive constant.
\end{thm}

\begin{proof}By Cauchy-Schwarz  inequality, we have	
	\begin{align}\label{CS}
	&\frac1{N}\sum_{m=1}^{N}\Big|\frac1{N}\sum_{n=1}^{N} \bnu(n) \bnu(n+m)\Big|\\
	&\leq
	\Big(\frac1{N}\sum_{m=1}^{N}\Big|\frac1{N}\sum_{n=1}^{N} \bnu(n) \bnu(n+m)\Big|^2\Big)^{\frac{1}{2}},
	\end{align}
	and by Bourgain's observation \cite[equations (2.5) and (2.7)]{BourgainD}, we have
	\begin{align}\label{bo2}
	&\sum_{m=1}^{N}\Big|\frac1{N}\sum_{n=1}^{N} \bnu(n)\bnu(n+m) \lambda^{n+m}\Big|^2 \nonumber\\
	&= \sum_{m=1}^{N}\Big|\int_{\T}\Big(\frac1{N}\sum_{n=1}^{N}
	\bnu(n)z^{-n}\Big)\Big(\sum_{p=1}^{2N} \bnu(p) {\big(\lambda.z\big)}^p\Big) z^{-m} dz\Big|^2 \nonumber \\
	&\leq  \int_{\T}\Big|\frac1{N}\sum_{n=1}^{N} \bnu(n)
	z^{-n}\Big|^2\Big|\sum_{p=1}^{2N} \bnu(p)
	{\big(\lambda.z\big)}^p\Big|^2 dz\\
	&\leq \sup_{z \in \T}\Big(\Big|\frac1{N}\sum_{n=1}^{N} \bnu(n)
	z^{-n}\Big|\Big)^2  \int_{\T}\Big|\sum_{p=1}^{2N} \bnu(p)
	{\big(\lambda.z\big)}^p\Big|^2 dz.
	\end{align}
	The  inequality (1.3) is due to Parseval inequality. Indeed, by putting
	\begin{align*}
	\Phi_N(z)=\Big(\frac1{N}\sum_{n=1}^{N}
	\bnu(n) z^{-n}\Big)\Big(\sum_{p=1}^{2N} \bnu(p)
	{\big(\lambda.z\big)}^p \Big).
	\end{align*}
	We see that for any $m \in \Z$,
	\begin{align*}
	&\widehat{\Phi_N}(m)\\
	&=\int_{\T}\Big(\frac1{N}\sum_{n=1}^{N}
	\bnu(n)z^{-n}\Big)\Big(\sum_{p=1}^{2N} \bnu(p)
	{\big(\lambda.z\big)}^p\Big) z^{-m} dz.
	\end{align*} 
	and
	\begin{eqnarray*}
		&&\sum_{m=1}^{N}\Big|\int_{\T}\Big(\frac1{N}\sum_{n=1}^{N}
		\bnu(n) z^{-n}\Big)\Big(\sum_{p=1}^{2N} \bnu(p)
		{\big(\lambda.z\big)}^p\Big) z^{-m} dz\Big|^2\\
		&=&\sum_{m=1}^{N}\Big|\widehat{\Phi_N}(m)\Big|^2\\
		&\leq& \int_{\T} |\Phi_N(z)|^2 dz.
	\end{eqnarray*}
	Now, by appealing to Davenport Theorem, we get
	\begin{align}\label{bo3}
	&\sum_{m=1}^{N}\Big|\frac1{N}\sum_{n=1}^{N} \bnu(n)\bnu(n+m) \lambda^{n+m}\Big|^2 \nonumber\\
	&\leq 
	\frac{C_\epsilon}{\log(N)^{\epsilon}}
	\end{align} 
	From this, we obtain the desired inequality and the proof is complete.
\end{proof}

\end{document}